\documentclass[twoside,reqno]{amsart}

\usepackage[english]{babel} 
\usepackage{exscale}   
\usepackage{amsmath}
\usepackage{amsfonts}
\usepackage{amssymb}
\usepackage{amsxtra}
\usepackage{stmaryrd}
\usepackage{xspace}

\usepackage{epsfig}
\usepackage{mathrsfs}
\usepackage{pstricks}
\usepackage{enumitem}
\usepackage{latexsym}
\usepackage{graphics}
\usepackage{xypic}

\newtheorem{Theorem}{Theorem}[section]
\newtheorem{Proposition}[Theorem]{Proposition}
\newtheorem{Corollary}[Theorem]{Corollary}
\newtheorem{Lemma}[Theorem]{Lemma}
\theoremstyle{Definition} 
\newtheorem{Definition}[Theorem]{Definition}
\theoremstyle{Definition}
\theoremstyle{Remark}
\newtheorem{Example}[Theorem]{Example}
\newtheorem{Remark}[Theorem]{Remark}

\newcommand{\Cset}{\mathbb{C}}

\newcommand{\Nset}{\mathbb{N}}

\newcommand{\Rset}{\mathbb{R}}

\newcommand{\CA}{\ensuremath{{\mathcal A}}\xspace}         
\newcommand{\CB}{\ensuremath{{\mathcal B}}\xspace}         
\newcommand{\CC}{\ensuremath{{\mathcal C}}\xspace}         
\newcommand{\CD}{\ensuremath{{\mathcal D}}\xspace}         
\newcommand{\CF}{\ensuremath{{\mathcal F}}\xspace}         
\newcommand{\CH}{\ensuremath{{\mathcal H}}\xspace}         
\newcommand{\CK}{\ensuremath{{\mathcal K}}\xspace}         
\newcommand{\CI}{\ensuremath{{\mathcal I}}\xspace}         
\newcommand{\CT}{\ensuremath{{\mathcal T}}\xspace}         

\newcommand{\eins}{\ensuremath{{\rm 1\kern-.25em l}}\xspace} 
\newcommand{\id}{\operatorname{id}}

\newcommand{\Trace}{\operatorname{Tr}}

\newcommand{\supp}{\mathop{\rm supp}\nolimits}

\begin{document}

\title[Universal Preparability of States 
and Asymptotic Completeness]
{Universal Preparability of States 
\\ and Asymptotic Completeness
}
\author{Rolf Gohm}
\author{Florian Haag}
\author{Burkhard K\"{u}mmerer}
\address{Rolf Gohm \\
Department of Mathematics \\ 
IMPACS, Aberystwyth University \\
Aberystwyth SY23 3BZ, United Kingdom}
\address{Florian Haag \\
Adlerstr.\,14, 73540 Heubach}
\address{Burkhard K\"{u}mmerer \\
Fachbereich Mathematik, Technische Universit\"at Darmstadt \\
Schlo{\ss}gartenstr. 7, \\
64289 Darmstadt, Germany}
\email{rog@aber.ac.uk}
\email{fu.haag@web.de}
\email{kuemmerer@mathematik.tu-darmstadt.de}
\subjclass[2000]{46L50, 46L51, 46L55, 54H20}
\keywords{repeated interaction, transition, preparability, tightness, stationary Markov chain, asymptotic completeness, controllability, observability, micromaser}
\date{}

\begin{abstract}
We introduce a notion of universal preparability for a state of a system, more precisely: for a normal state on a von Neumann algebra. It describes a situation where from an arbitrary initial state it is possible to prepare a target state with arbitrary precision by a repeated interaction with a sequence of copies of another system. For $\CB(\CH)$ we give criteria sufficient to ensure that all normal states are universally preparable which can be verified for a class of non-commutative birth and death processes realized by the interaction of a micromaser with a stream of atoms. As a tool the theory of tight sequences of states and of stationary states is further developed and we show that in the presence of stationary faithful normal states universal preparability of all normal states is equivalent to asymptotic completeness, a notion studied earlier in connection with the scattering theory of non-commutative Markov processes.
\end{abstract}
\maketitle

\section{Introduction}
\label{section:intro}

The present paper discusses preparability of states as an asymptotic property of quantum Markov processes and applies the resulting theory to some physical systems of experimental interest. From a system theoretic point of view these processes model a repeated interaction of one system with a sequence of copies of another system and hence we prove controllability of certain quantum systems and contribute to quantum control theory. 
In this introduction we first give some intuitive background before summarizing the contents of the paper. 
 
A typical quantum Markov process in discrete time is obtained by tensoring an initial system with observable algebra $\CA$ to an infinite tensor product of copies of another algebra $\CC$. Its Markov dynamics is obtained as the product of the tensor shift on the infinite tensor product of $\CC$ with a \emph{coupling automorphism} $\alpha$ acting non-trivially only on the tensor product of \CA with one of the copies of $\CC$. The Markovian semigroup $(T^n)$ is then obtained by applying the conditional expectation onto \CA, with respect to a product state on the copies of $\CC$, to elements having started in $\CA$ and evolved for
$n$ time steps. Such processes have been introduced in \cite{Ku85}, some overview is given in \cite{Ku06}.

A paradigmatic example from physics
is the experimental setting of a micromaser, which is addressed in Section \ref{section:example} below: A stream of two-level atoms passes through a cavity, one after the other. While being inside the cavity an atom interacts with one mode of the electromagnetic field. In this case $\CA = \CB(\CH)$ stands for the observables of the field mode, the algebra \CC given by the $2\!\times\!2$-matrices represents a two-level atom and the tensor product of copies of $\CC$ represents a stream of such atoms. The coupling automorphism $\alpha$ describes the overall effect of the interaction between the field mode and a single atom while passing through the cavity (\cite{WBKM00}, \cite{Ku06}). 

As a basic building block this setting also appears in other
experiments, such as the famous realization of quantum feedback described in \cite{S-H11} and related to the physics Nobel prize 2012 for S.\,Haroche. See also \cite{Ro14} for a recent survey on Markovian models, feedback and reservoir engineering in the context of the experimental progress. In our paper we do not study feedback but we develop a rigorous theory of coherent open-loop control of the system $\CA$, for a given coupling with another system $\CC$. The input states can vary and they may be entangled between different copies of $\CC$. To work out further connections of our setting with recent developments in physics will be a rewarding task for the future.

Our starting point here is that such a type of Markovian dynamics
suggests to consider it from the point of view of scattering theory: the shift takes the role of a free dynamics which is locally perturbed by the coupling automorphism $\alpha$. This was begun in \cite{KM00} where the notion of \emph{asymptotic completeness} for such systems was introduced. It roughly means that observables in \CA asymptotically end up in the tensor products of $\CC$ (see Definition \ref{def:ac} for the precise notion). Asymptotic completeness and scattering theory for Markov processes have been further discussed in \cite{WBKM00}, \cite{Go04,Go04b}, and \cite{GKL06}.

There is a dual point of view of asymptotic completeness which is at the core of the present paper: If for large times the dynamics drives the algebra $\CA$ completely into the tensor product of copies of $\CC$, then one can hope that a prescribed target state $\rho$ on $\CA$ can be prepared by preparing a state $\theta$ on this tensor product of $\CC$ such that its restriction to the time shifted copy of $\CA$ is close to $\rho$. Then the time evolution on the states (i.\,e., in the Sch\"{o}dinger picture) should drive an arbitrary initial state $\sigma$ on $\CA$ close to the target state $\rho$. This property is called \emph{universal preparability} in this paper (cf. its rigorous definition in \ref{def:preparation}(c)). In the physical applications we have in mind (the micromaser is an example) states on system $\CA$ cannot be directly accessed by experiment, while one can choose as $\CC$ a system whose states can be manipulated more easily. This intuitive idea is mathematically verified in the equivalence between (a) and (b1) of our main Theorem \ref{th:main}.

From this dual point of view and in a system theoretic language asymptotic completeness means controllability of the system \CA via input states on a tensor product of copies of $\CC$. Alternatively it can be discussed from the point of view of coding theory as was done in \cite{GKL06}, and some ideas from there are further developed in the present paper.
It may be viewed as one of the outcomes of this paper -- some indications were already seen in \cite{GKL06} -- that asymptotic completeness is a topological notion, encoded in the property of \emph{topological transitivity} (cf. Definition \ref{def:preparation}), rather than a measure theoretic one (cf. Theorem \ref{th:main}(c)). As a consequence stationary states enter the discussion only in later sections. Nevertheless, asymptotic completeness needs the topologies of von Neumann algebras. 

Our starting point here was to find a proof of asymptotic completeness of the micromaser dynamics which is not ad hoc designed for this particular system but instead is embedded into a systematic approach to asymptotic completeness. Such an approach has been started in the dissertation \cite{Ha06} by one of the authors and is further developed in the present paper. Asymptotic completeness of the micromaser and related systems now follows easily in Theorem \ref{th:micromaser} from the main Theorem \ref{th:main} in combination with Theorem \ref{th:reverse}.

It turned out that for many parts of our discussion it suffices to concentrate on one-sided time evolutions. This amounts to generalize a coupling automorphism $\alpha$ of $\CA \otimes \CC$ to a $*$-homomorphism $J:\CA \to \CA \otimes \CC$ on von Neumann algebras, which we call a \emph{transition}. In the presence of $\alpha$ it is given by $J(a)=\alpha(a\otimes \eins)$ for $a \in \CA$ (cf. Definition \ref{def:transition} and the discussion thereafter). Moreover, it suffices to consider only finitely many time steps and thus infinite tensor products are avoided in this paper. We think that this also emphasizes the practical applicability of our results for the design of physical experiments.

Let us now summarize the contents of this paper.

In Section \ref{section:preparation} we introduce the basic notions of a transition $J$, taken from \cite{GKL06}, of (universal) $J$-preparability, and of topological transitivity. We develop their basic theory and find in Theorem \ref{th:up} a useful sufficient criterion 
for universal $J$-preparability of all normal states on $\CB(\CH)$, the bounded linear operators on a Hilbert space $\CH$. This criterion gives a first hint why universal preparability is more common than one might originally think when confronted with the definition.

In Section \ref{section:criterion} we show that the concatenation of preparation procedures works well in our setting and we use this tool to prove another sufficient criterion 
for universal $J$-preparability of all normal states on $\CB(\CH)$ which is easier to verify in practice, see Theorem \ref{th:reverse}. It relies on the intuitive idea that if a vector state, in concrete realizations given as a ground state, 
is universally $J$-preparable not only in the forward but also in the reverse time direction, then we can go between any two normal states via the vector state by concatenating these procedures in a suitable way. For earlier versions of this idea in related contexts see also
\cite{WBKM00,BG07}. 

For a deeper analysis of $J$-preparability the theory of stationary states for positive operators is needed and for this reason we review in Section \ref{section:stationary states}
some relevant parts of this theory and develop it further. In particular for infinite dimensional systems it is necessary to develop a non-commutative version of the probabilistic concept of a tight sequence of probability measures. Such a version has been defined first in \cite{FR01}, we add to that a non-commutative version of Prokhorov's theorem, Theorem \ref{th:prok}, and a number of further connections between tightness and stationarity. In Section \ref{section:stationary states} we also review some theory about absorbing states. 

In Section \ref{section:stationary markov} we consider stationary states for transitions and the corresponding stationary Markov processes and we review and develop the theory of the dual extended transition operator from \cite{Go04} and \cite{GKL06}. Introducing the concept of a tight transition it turns out that tightness in this sense is satisfied in many situations and we thus obtain a large class of transitions for which the following analysis is applicable. We finally review the concept of asymptotic completeness.

Now we have all pieces together to prove in Section \ref{section:ac-up} the main result of this paper, Theorem \ref{th:main}. It states that for a tight transition $J$ on $\CB(\CH)$ asymptotic completeness is equivalent to the universal $J$-preparability of all normal states and is also equivalent to topological transitivity of $J$. The first equivalence brings together the scattering theory of Markov processes and the issue of state preparation by repeated interactions discussed in the first part of the paper. The second equivalence makes clear that stationary states should be seen as a tool and that we actually deal with a deeper topological property here. We also give another system theoretic point of view by proving the equivalence with an observability property for the time reversed system, see Theorem \ref{thm:observable}.

In the final Section \ref{section:example} we show that our theory can be applied to a class of non-commutative birth and death processes discussed in \cite{BGKRSS}, which contains the micromaser as a special case (Theorem \ref{th:micromaser}).

We add some remarks on our notational conventions.
Because all main results concern von Neumann algebras and, even more specific, the von Neumann algebra $\CB(\CH)$ of bounded linear operators on a Hilbert space $\CH$, we simplify the terminology by restricting to such a setting from the beginning.
We refer to \cite{KR86} and \cite{Tak79} for definitions and facts about operator algebras. 
For a von Neumann algebra $\CA$ we denote by $\CA_*$ its predual, the Banach space of normal linear functionals on $\CA$. All tensor products are tensor products of von Neumann algebras.
Any weak$^*$-continuous map $T$ on $\CA$ is also called normal and its preadjoint on $\CA_*$ is denoted by $T_*$. In the case of $\CA = \CB(\CH)$ the predual can be identified with $\CT(\CH)$, the Banach space of trace class operators on $\CH$. If we refer to a projection $p \in \CA$ we always mean an orthogonal projection. Given a normal state $\varphi$ on $\CA$ we denote by $\supp\varphi$ its support projection.
If $0 \not=\xi \in \CH$ then we denote by $\omega_\xi = \langle \xi, \, \cdot \, \xi\rangle$ the corresponding functional on $\CB(\CH)$ and by $p_\xi:=\supp\omega_\xi$ its (one-dimensional) support projection. The inner product is anti-linear in the first and linear in the second component. We always assume the Hilbert spaces and the preduals of von Neumann algebras to be separable, so in particular there always exist faithful normal states.


\section{A Protocol of State Preparation based on Transitions}
\label{section:preparation}

We begin by introducing the concept of a transition, cf. \cite{GKL06}.
A transition constitutes the fundamental building block of the type of dynamics which is studied in this paper. 
In quantum physics such dynamics describe repeated interactions between systems.


\begin{Definition}[\cite{GKL06}] \label{def:transition}
Let $\CA$ and $\CC$ be von Neumann algebras. An injective unital normal $*$-homomorphism $J: \CA \rightarrow \CA \otimes \CC$ is called a transition.
\end{Definition}

If $\alpha$ is a $*$-automorphism of $\CA \otimes \CC$ then we get a transition by $J(a) := \alpha(a \otimes \eins)$. (Normality is automatic here, see \cite{Tak79}, III.3.10.) In this case we say that the transiton $J$ is obtained from a \emph{coupling automorphism} $\alpha$. In particular, if $u \in \CA \otimes \CC$ is unitary we can define a transition $J(a) := u^* \; a \otimes \eins \; u$. In the general case we may still think of $J$ as describing on observables in $\CA$ the effect of one step of a time evolution which is produced by an interaction between two systems with observable algebras $\CA$ and $\CC$. This is the Heisenberg picture. 

In the corresponding Schr\"{o}dinger picture the preadjoint $J_*: (\CA \otimes \CC)_* \rightarrow \CA_*$ maps normal states on $\CA \otimes \CC$ to normal states on $\CA$. In particular, if we have at time $0$ a normal product state $\sigma \otimes \theta$ on $\CA \otimes \CC$ then after the interaction between the two systems has taken place the state on $\CA$ has changed to $J_*(\sigma \otimes \theta) = (\sigma \otimes \theta) J$. Hence studying transitions means to concentrate on the change of state of the system $\CA$ in dependence of initial states on $\CA$ and $\CC$ 
(and the interaction between $\CA$ and $\CC$). More specifically, if $\CA$ and $\CC$ are algebras of functions on finite sets $A$ and $C$ then a transition is induced by a map $\gamma: A \times C \to A$: Depending on the actual state $c\in C$ a state $a\in A$ moves to $\gamma(a,c) \in A$. This explains the term \emph{transition} and was the starting point of \cite{GKL06}. 
It is the main objective of this paper to ask which state changes on $\CA$ can be triggered by choosing suitable states on $\CC$ when a transition $J$ is given (cf. Definition \ref{def:preparation} below).

Considering only one time step is not enough, however. 
From a transition $J: \CA \rightarrow \CA \otimes \CC$ we can construct a repeated interaction of the system described by $\CA$ with a sequence of copies of systems described by $\CC$, as follows (cf. \cite{GKL06}): Let us denote the copies of $\CC$ by $\CC_{(1)}, \CC_{(2)}, \ldots$ and the
corresponding copies of $J$ by $J_{(1)}, J_{(2)}, \ldots$, so $J_{(n)}: \CA \rightarrow \CA \otimes \CC_{(n)}$. Then we get the time evolution for the repeated interactions up to time $n$ as a composition
\[
J_n := J_{(1)} J_{(2)} \ldots J_{(n)} \colon \quad \CA \rightarrow \CA \otimes \CC_n \quad \text{with}\; \CC_n := \bigotimes^n_{j=1} \CC_{(j)}\,; 
\]
here the standard unital embeddings of the $\CC_{(j)}$ into $\CC_n = \bigotimes^n_{j=1} \CC_{(j)}$ are used but omitted from the notation. Note that $J_{n*} := J_{(n)*} J_{(n-1)*} \ldots J_{(1)*}$,
so this really describes the change of state on $\CA$ after we interacted at time $1$ with the system described by $\CC_{(1)}$,
etc., finally at time $n$ with the system described by $\CC_{(n)}$.

Suppose, in particular, that the transition $J$ comes from a coupling automorphism $\alpha: \CA \otimes \CC \rightarrow \CA\otimes \CC$ such that $J(a) := \alpha(a \otimes \eins)$.  Then we have copies $\alpha_{(j)}$ on $\CA \otimes \CC_{(j)}$ and an automorphism $\alpha_n = \alpha_{(1)}
\ldots \alpha_{(n)}$ of $\CA \otimes \CC_n$. The restriction of $\alpha_n$ to $\CA \otimes \eins$ induces $J_n$. 

In this case there exists a reverse transition $J^r$ given by $J^r(a) := \alpha^{-1}(a \otimes \eins)$. Note that to get an inverse of $\alpha_n$ some reordering of the tensor positions must be applied: $(\alpha_n)^{-1} = (\id \otimes R_n) \circ
(\alpha^{-1})_n \circ (\id \otimes R_n)$, where $\id$ stands for the identity on $\CA$ and $R_n$ is the automorphism of $\CC_n$ which interchanges the positions in 
$\CC_n = \bigotimes^n_{j=1} \CC_{(j)}$ by the rule $j \leftrightarrow n-j+1\,$ for $j=1,\ldots,n$. We can think of $J^r$ as a time reversal of $J$. 

In \cite{GKL06} the sequence $(J_n)_{n\in\Nset}$ is interpreted as a non-commutative topological Markov chain. The Markovian character will be further illustrated when we later discuss stationary states. But for the following protocol of state preparation on $\CA$ we only need a transition $J: \CA \rightarrow \CA \otimes \CC$.

With Definition \ref{def:preparation} below we are led to the main questions discussed in this paper. Loosely speaking, given a transition $J$ as above we would like to make use of it to prepare a target state $\rho$ on $\CA$ from an arbitrary initial state $\sigma$ on $\CA$  by preparing a suitable state $\theta$ on $\CC$. Except for trivial cases, however, we cannot hope to do this within one time step.
If we have a normal state $\theta$ on $\CC_n = \bigotimes^n_{j=1} \CC_{(j)}$ then an initial state $\sigma$ on $\CA$, originally combined with $\theta$ as a product state $\sigma \otimes \theta$,
is changed after $n$ steps into the normal state $(\sigma \otimes \theta) J_n$ on $\CA$. We can interpret this as a 
protocol for the preparation of states on $\CA$ and in this paper we are particularly interested in the potential of such a protocol for producing from an arbitrary initial state on $\CA$ any desired target state on $\CA$ with any prescribed precision. We introduce some terminology.


\begin{Definition} \label{def:preparation}
Let $J: \CA \rightarrow \CA \otimes \CC$ be a transition. 
\begin{itemize}
\item[(a)]
A normal state $\rho$ on $\CA$ is called $J$-preparable from a normal state $\sigma$ on $\CA$
if there exists a sequence $(\theta_k)_{k \in \Nset}$, where each $\theta_k$ is a normal state on an algebra $\CC_{n_k}$ (with $n_k \in \Nset$)
such that the sequence $\big((\sigma \otimes \theta_k) J_{n_k}\big)_{k \in \Nset}$ converges weakly to $\rho$. Such a sequence $(\theta_k)$ is called a preparing sequence
(from $\sigma$ to $\rho$).
\item[(b)]
If every normal state $\rho$ is
$J$-preparable from all normal states $\sigma$ on $\CA$ then we call $J$ topologically transitive. 
\item[(c)]
A normal state $\rho$ on $\CA$ is called universally $J$-preparable if there exists a sequence $(\theta_k)_{k \in \Nset}$, where each $\theta_k$ is a normal state on an algebra $\CC_{n_k}$ (with $n_k \in \Nset$)
such that for all normal states $\sigma$ on $\CA$ the sequence
$\big((\sigma \otimes \theta_k) J_{n_k}\big)_{k \in \Nset}$ converges weakly to $\rho$.
Such a sequence $(\theta_k)$ is called a universally preparing sequence (for $\rho$).
\end{itemize}
\end{Definition}

\begin{Remark} \normalfont
We add a few comments on these definitions. 

\begin{enumerate}

\item
Spelling out the definition of weak convergence, we see that it is equivalent to \ref{def:preparation}(a) to say that $\rho$ is $J$-preparable from $\sigma$
if we can find a sequence of sizes $(n_k)$ for the tensor product algebras $\CC_{n_k}$ and normal states $\theta_k$ on them, such that for any given $x \in \CA$ and $\epsilon > 0$
we have $| (\sigma \otimes \theta_k) J_{n_k} (x) - \rho(x) | < \epsilon$ if $k$ is chosen big enough. In other words, if we are able to provide these states $\theta_k$ then by repeatedly applying the transition $J$ we can change the state in $\CA$ from $\sigma$ to $\rho$, with arbitrary precision. We accept that this may only be an approximation and we may in some cases need $n_k \to \infty$ to achieve convergence but the definition is flexible enough to include more elementary cases as well: If already $(\sigma \otimes \theta) J_n = \rho$ for a normal state $\theta$ on $\CC_n$ for some finite $n$ then $\rho$ is $J$-preparable from $\sigma$ in the sense of \ref{def:preparation}(a). In fact, in this case we can choose a constant sequence $(\theta_k)$ with $\theta_k := \theta$ for all $k$ (here we have $n_k = n$ for all $k$). The reader should think of \ref{def:preparation}(a) and \ref{def:preparation}(c) as quantum versions of the notion of approximate controllability in classical control theory, cf. \cite{Co07}.

\item It is clear that if all normal states are universally $J$-preparable then $J$ is topologically transitive. We shall see later in our main Theorem \ref{th:main} that for $\CA = \CB(\CH)$ and under an additional assumption on $J$ (tightness) the converse also holds

\item Definition \ref{def:preparation}(b) shows that our originally physically motivated investigation is also of interest in the theory of dynamical systems and in ergodic theory. To connect with the general theory of topologically transitive spaces of operators, see \cite{DMR08}, we can consider the space of all operators on the predual $\CA_*$ which have the form $\sigma \mapsto (J_n)_* (\sigma \otimes \theta)$.

\item If there is a universally preparing sequence then one can prepare a state $\rho$ of a system $\CA$ by preparing a certain sequence of normal states on the algebras $\CC_n$ even if no information on the initial state $\sigma$ of $\CA$ is available. The micromaser discussed in Section \ref{section:example} is a typical example of such a system (cf. \cite{WBKM00}).

\item Note that the preparing sequences in the sense of Definition \ref{def:preparation} are not uniquely determined and in fact there are a lot of additional issues an experimentalist may care about, for example finding good solutions within a reasonably small time period or avoiding experimental difficulties involved in preparing specific classes of states (note that we allowed the use of arbitrarily entangled states $\theta_k$ on $\CC_{n_k}$ as preparing states). More generally one may ask in the design of preparing states for efficiency or optimality with respect to various criteria.
We don't discuss these important issues in this paper and there is a lot of work to be done in this respect.

\end{enumerate}

\end{Remark}

The choice of the weak topology for the predual agrees in applications with physicists' focus on expectations for observables in $\CA$. But in many physically relevant cases we can replace the weak topology by the norm topology. Recall that a von Neumann algebra is called atomic if every nonzero projection majorizes a nonzero minimal projection. Moreover, if a sequence of normal states on an atomic von Neumann algebra converges weakly to a normal state then it converges in norm, see \cite{Tak79}, III.V.11; therefore, Definition \ref{def:preparation} can be strengthened as follows.


\begin{Proposition} \label{prop:atomic}
If $\CA$ is atomic, in particular if $\CA = \CB(\CH)$, we can replace the weak convergence in Definition \ref{def:preparation} by norm convergence.
\end{Proposition}

For the case $\CA = \CB(\CH)$ for a Hilbert space $\CH$ this also follows from our Lemma \ref{lem:compact} below. 

From a mathematical point of view the existence of universally preparing sequences may be surprising at first sight. Our main result of this section, Theorem \ref{th:up}, explains how they arise naturally in some situations. The rest of this section prepares its proof. This theorem will in turn play a crucial role in the proof of the main result of this paper, Theorem \ref{th:main}.


\begin{Proposition} \label{prop:closed}
Let $J: \CA \rightarrow \CA \otimes \CC$ be a transition, further let $\rho$ be a normal state on $\CA$ and $(\theta_k)$ a sequence
where each $\theta_k$ is a normal state on an algebra $\CC_{n_k}$.
Denote by $\CI = \CI(\rho, (\theta_k))$ the set of all normal states
$\sigma$ on $\CA$ such that $\rho$ is $J$-preparable from $\sigma$ with preparing sequence $(\theta_k)$.
Then $\CI$ is convex and closed, weakly or w.r.t. the norm (it might be empty, however).
\end{Proposition} 

\begin{proof}
The convexity of $\CI$ is clear because the result of a preparation with the preparing sequence $(\theta_k)$ depends linearly on the initial state $\sigma$. For convex sets weak and norm closure coincide, so it is enough to show that $\CI$ is norm closed. 
Let $a$ be any element in the unit ball of $\CA$ and $\epsilon > 0$.
 If $\sigma$ is in the closure of $\CI$ then find $\sigma' \in \CI$ such that $\| \sigma - \sigma' \| < \frac{1}{2} \epsilon$ and $k_0 \in \Nset$ so that $| (\sigma' \otimes \theta_k) J_{n_k}(a) - \rho(a)|< \frac{1}{2} \epsilon$ for all $k \ge k_0$. Then for $k \ge k_0$ we also have
$|(\sigma \otimes \theta_k) J_{n_k}(a) - \rho(a)| < \epsilon$. Hence $\sigma \in \CI$.
\end{proof}


The next result shows that if a sequence of states prepares a vector state $\omega_\xi$ on $\CB(\CH)$ from a faithful normal state $\varphi$ then the same sequence prepares $\omega_\xi$ from any normal state.

\begin{Proposition} \label{prop:faithful}
Suppose $\CA = \CB(\CH)$, where $\CH$ is a Hilbert space, and let
$J: \CB(\CH) \rightarrow \CB(\CH) \otimes \CC$ be a transition.
Let $\varphi$ be a faithful normal state on $\CB(\CH)$ and $\omega_\xi$ a vector state from the unit vector $\xi \in \CH$.
If $\omega_\xi$ is $J$-preparable from $\varphi$ with a preparing sequence $(\theta_k)$ then $\omega_\xi$ is universally $J$-preparable with universally preparing sequence $(\theta_k)$.
\end{Proposition}

\begin{proof}
Any normal state $\sigma$ on $\CB(\CH)$ can be approximated by states with finite dimensional support, this is obvious if we think of these states as density matrices, i.e., $\rho(\cdot) = \Trace(d\,\cdot)$
for a positive trace class operator $d$ with $\Trace(d)=1$. 
Using Proposition \ref{prop:closed} it is therefore enough to show that $\omega_\xi$ is $J$-preparable with preparing sequence $(\theta_k)$ from any initial state $\sigma$ with finite dimensional support $q$. For such a state $\sigma$  and the given faithful normal state $\varphi$ the function $a \mapsto \frac{\sigma(a)}{\varphi(a)}$ is continuous and positive on $\{a \in q\, \CB(\CH) q \colon a \ge 0,\, \|a\|=1\}$, the positive part of the unit sphere of the finite dimensional algebra $q\, \CB(\CH) q$. Because this is a compact set the function attains a finite maximum $c > 0$ there and it follows that $\sigma \le c \, \varphi$. 

Let $p_\xi$ be the (one-dimensional) support
projection of the vector state $\omega_\xi$. Then we have
\[
0 \le (\sigma \otimes \theta_k) J_{n_k} (\eins - p_\xi) \le
c (\varphi \otimes \theta_k) J_{n_k} (\eins - p_\xi) \to c \,\omega_\xi(\eins - p_\xi) = 0 \quad (k\to\infty).
\]
Hence $(\sigma \otimes \theta_k) J_{n_k} (p_\xi) \to 1$ if
$k\to\infty$. This implies 
$ \lim_{k\to\infty} \| (\sigma \otimes \theta_k) J_{n_k} - \omega_\xi \| = 0$ (see for example \cite{Go04}, A.5.3, for a worked out argument concerning the last step). 
\end{proof}


\begin{Definition} \label{def:compatible}
Let $J: \CA \rightarrow \CA \otimes \CC$ be a transition and let
$(\theta_k)$ and $(\theta'_k)$ be preparing sequences of any two states in the sense of Definition \ref{def:preparation} so that $\theta_k$ is a state on $\CC_{n_k}$ and $\theta'_k$ is a state on $\CC_{n'_k}$. If
$n_k = n'_k$ for all $k \in \Nset$ then $(\theta_k)$ and $(\theta'_k)$
are called compatible.
\end{Definition}

If we have compatible preparing sequences then superposition becomes available as an additional tool. Recall that a $\sigma$-convex combination of elements $x_1, x_2, \ldots$ is an expression 
$\sum^\infty_{i=1} c_i x_i$ with $0 \le c_i \in \Rset$ for all $i$ and $\sum^\infty_{i=1} c_i = 1$.


\begin{Proposition} \label{prop:convex}
Let $J: \CA \rightarrow \CA \otimes \CC$ be a transition.
If we have a compatible set of universally preparing sequences for a set of normal states on $\CA$ then any $\sigma$-convex combination
of such sequences is a compatible universally preparing sequence for the corresponding $\sigma$-convex combination of states.
\end{Proposition}

\begin{proof}
Consider a $\sigma$-convex combination $ \tau := \sum^\infty_{i=1} c_i \tau_i$ where all $\tau_i$ are universally $J$-preparable 
with compatible preparing sequences $(\theta^{(i)}_k)$. For all $k \in \Nset$ define
$\theta_k := \sum^\infty_{i=1} c_i \theta^{(i)}_k$.
We check that $\tau$ is universally $J$-preparable with the compatible universally preparing sequence $(\theta_k)$. Indeed, fix a normal state $\sigma$ on $\CA$ and an element
$a$ in the unit ball of  $\CA$. Then we can, for any $\delta > 0$, find $i_0$ big enough so that $\sum_{i > i_0} c_i < \delta$ and then, for any $\epsilon > 0$, find $k_0$
big enough so that for all $k > k_0$ and all $i=1, \ldots, i_0$ 
\[
|(\sigma \otimes \theta^{(i)}_k) J_{n_k}(a) - \tau_i(a)| < \epsilon.
\]
Then for $k > k_0$ we have 
$|(\sigma \otimes \theta_k) J_{n_k}(a) - \tau(a)| < \epsilon + 2 \,  \delta$.
\end{proof}

Remark: A physicist who can prepare states $\rho_1$ and $\rho_2$ 
in some way can also prepare convex combinations by performing the corresponding procedures with corresponding probabilities. But this is not $J$-preparation of a convex combination in the sense of Definition \ref{def:preparation}. Proposition \ref{prop:convex} gives us a sufficient condition for $J$-preparability of a convex combination.
\\

We are ready for the main conclusion from these considerations which gives us a useful sufficient condition for universal preparability of all normal states on $\CA = \CB(\CH)$. It will be used later to prove one of our main results, Theorem \ref{th:main}.


\begin{Theorem} \label{th:up}
Suppose $\CA = \CB(\CH)$, where $\CH$ is a Hilbert space, and let
$J: \CB(\CH) \rightarrow \CB(\CH) \otimes \CC$ be a transition.
If there exists a faithful normal state $\varphi$ on $\CB(\CH)$ so that
all vector states are $J$-preparable from $\varphi$ by compatible preparing sequences then all normal 
states on $\CB(\CH)$ are universally $J$-preparable with compatible universally preparing sequences.
\end{Theorem}

\begin{proof}
Applying Proposition \ref{prop:faithful} we find that all vector states are universally $J$-preparable by compatible universally preparing sequences. The representation of an arbitrary normal state by a density matrix shows that all normal states on $\CB(\CH)$ can be written as $\sigma$-convex combinations of vector states and hence the theorem follows from Proposition \ref{prop:convex}.
\end{proof}


\section{Time Reversal Criterion for Universal Preparability}
\label{section:criterion}

In this section we give another sufficient condition for universal preparability
which is based on the idea of time reversal, see Theorem \ref{th:reverse}.
Because we want to concatenate preparation procedures we start by discussing the transitivity properties of $J$-preparability in more detail. Keeping track of the positions in the tensor products makes the notation somewhat laborious although the geometric ideas are simple.


\begin{Lemma} \label{lem:composition}
Let $J: \CA \rightarrow \CA \otimes \CC$ be a transition.
Further let $\theta$ be a state on 
$\CC_n = \bigotimes^n_{j=1} \CC_{(j)}$
and $\chi$ be a state on 
$\CC_m = \bigotimes^m_{j=1} \CC_{(j)}$.
Now we form a product state $\theta \otimes \chi^n$ on
$\CC_{n+m} = \bigotimes^{n+m}_{j=1} \CC_{(j)}$ where
$\chi^n$ is the state $\chi$ but shifted by $n$ positions in the tensor product to $\bigotimes^{n+m}_{j=n+1} \CC_{(j)}$.
Then for any state $\sigma$ on $\CA$ we have
\[
\big(\sigma \otimes (\theta \otimes \chi^n) \big) J_{n+m}
= (\big[(\sigma \otimes \theta) J_n \big] \otimes \chi) J_m \,.
\]
\end{Lemma}

\begin{proof}
Let $a \in \CA$. Suppose that $J_m(a) = \sum_i a_i \otimes c_i$ with $a_i \in \CA$ and $c_i \in \CC_m$ (in infinite dimensional algebras this may only be possible approximately but this is enough for the following argument). Then $J_{n+m} = J_nJ_{(n+1)}\dots J_{(n+m)}$ hence $J_{n+m}(a) = \sum_i J_n(a_i) \otimes c_i^n$ where $c_i^n$ is the element $c_i$ but shifted by $n$ positions in the tensor product. Therefore, we find
\begin{eqnarray*}
\big(\sigma \otimes (\theta \otimes \chi^n) \big) J_{n+m}(a)
&=& (\sigma\otimes \theta)\big(\sum_i J_n(a_i)\chi(c_i)\big)
= \big[(\sigma\otimes \theta) J_n\big]\big(\sum_i a_i\chi(c_i)\big)\\
&=& \big(\big[(\sigma \otimes \theta) J_n \big] \otimes \chi\big) J_m(a)\,.
\end{eqnarray*}
\end{proof}


The following result shows that preparability is transitive. Again the idea behind its formulation and its proof is simple: If  $\rho$ can be prepared from $\tau$ and $\tau$ from $\sigma$ then first move $\sigma$ close to $\tau$ and afterwards move $\tau$ close to $\rho$.

\begin{Proposition} \label{prop:composition}
Let $J: \CA \rightarrow \CA \otimes \CC$ be a transition
and let $\sigma, \tau, \rho$ be normal states on $\CA$.
If $\tau$ is $J$-preparable from $\sigma$ with preparing sequence $(\theta'_k)$ and $\rho$ is $J$-preparable from $\tau$ with preparing sequence $(\theta''_k)$ then $\rho$ is $J$-preparable from $\sigma$. 

Suppose that $\theta'_\ell$ is a state on $\CC_{n'(\ell)}$ and $\theta''_m$ is a state on $\CC_{n''(m)}$.
There is a preparing sequence $(\theta_k)$ from $\sigma$ to $\rho$ where all $\theta_k$ have the form
\[
\theta_k = \theta'_\ell \otimes (\theta''_m)^{n'(\ell)}
\]
(notation as in Lemma \ref{lem:composition})
for suitable $\ell, m$.
If $\CA$ is atomic, in particular if $\CA = \CB(\CH)$, then
any sequence $(\theta_k)$ of this form is a preparing sequence
provided both $\ell \to \infty$ and
$m \to \infty$ if $k \to \infty$. 
\end{Proposition}

\begin{proof} 
Given any weak neighbourhood $U_\rho$ of $\rho$ choose $m$ big enough so that we have $(\tau \otimes \theta''_m) J_{n''(m)} \in U_\rho$.
Because $\theta \mapsto J_{n''(m)*}(\theta \otimes \theta''_m)$ is a weakly continuous map on $\CA_*$ it maps a weak neighbourhood $U_\tau$ of $\tau$ into $U_\rho$. 
If we now choose $\ell$ big enough (depending on $m$) so that 
we have $(\sigma \otimes \theta'_\ell) J_{n'(\ell)} \in U_\tau$ then with Lemma \ref{lem:composition} also 
\[
\big(\sigma \otimes (\theta'_\ell \otimes (\theta''_m)^{n'(\ell)}) \big)
J_{n'(\ell)+n''(m)} \in U_\rho \,.
\]
If $\CA$ is atomic then by Proposition \ref{prop:atomic} we can replace weak convergence by norm convergence and obtain, again with Lemma \ref{lem:composition},
\begin{eqnarray*}
&& \| \big(\sigma \otimes (\theta'_\ell \otimes (\theta''_m)^{n'(\ell)}) \big) J_{n'(\ell)+n''(m)} - \rho \| \\
&\le&
\| (\big[(\sigma \otimes \theta'_\ell) J_{n'(\ell)} \big] \otimes \theta''_m) J_{n''(m)} - (\tau \otimes \theta''_m) J_{n''(m)} \|
+ \| (\tau \otimes \theta''_m) J_{n''(m)} - \rho \| \\
&\le&
\| (\sigma \otimes \theta'_\ell) J_{n'(\ell)} - \tau \|
+ \| (\tau \otimes \theta''_m) J_{n''(m)} - \rho \|
\end{eqnarray*}
and this tends to $0$ if $\ell \to \infty$ and
$m \to \infty$.
\end{proof}


\begin{Corollary} \label{cor:composition}
Suppose $\CA$ is atomic and let $J: \CA \rightarrow \CA \otimes \CC$ be a transition. Suppose further that $\tau$ is universally $J$-preparable.  If $\rho$ is $J$-preparable from $\tau$ then $\rho$ is universally $J$-preparable. If $\rho_1, \rho_2$ are $J$-preparable from $\tau$ with compatible preparing sequences then $\rho_1, \rho_2$ are universally $J$-preparable
with compatible universally preparing sequences. 
\end{Corollary}

\begin{proof}
The construction of the universally preparing sequences can be done as in Proposition \ref{prop:composition}.
\end{proof}

For the rest of this section we assume that we have a transition $J: \CB(\CH) \rightarrow \CB(\CH) \otimes \CC$ which is induced by a coupling automorphism 
$\alpha$ of $\CB(\CH) \otimes \CC$, i.\,e., $J(a) = \alpha(a \otimes \eins)$ for $a \in \CB(\CH)$.
Then we also have another transition $J^r: \CB(\CH) \rightarrow \CB(\CH) \otimes \CC$ given for $a \in \CB(\CH)$ by $J^r(a) = \alpha^{-1}(a \otimes \eins)$ which can be thought of as a time reversal. Recall from the beginning of Section \ref{section:preparation} that $(\alpha_n)^{-1} = (\id \otimes R_n) \circ (\alpha^{-1})_n \circ (\id \otimes R_n)$, where $R_n$ interchanges the positions in 
$\bigotimes^n_{j=1} \CC_{(j)}$ by $j \leftrightarrow n-j+1\,$ for $j=1,\ldots,n$.

The next result shows that in this situation preparation of a vector state on $\CB(\CH)$ from an arbitrary normal state can be inverted. This will allow us in Theorem \ref{th:reverse} to use in certain situations a vector state as an intermediate step when preparing general normal states.


\begin{Proposition} \label{prop:reverse}
Let $\omega_\xi$ be a vector state on $\CB(\CH)$, with $\xi$ a unit vector in $\CH$. If $\omega_\xi$ is $J^r$-preparable from a normal state $\rho$ on $\CB(\CH)$ with a preparing sequence $(\theta^r_k)$ 
then, conversely, $\rho$ is $J$-preparable from $\omega_\xi$. A preparing sequence $(\theta_k)$ from $\omega_\xi$ to $\rho$ can be chosen compatibly with $(\theta^r_k)$, i.e., $\theta_k$ is a state on the same algebra $\CC_{n_k}$ as $\theta^r_k$ for all $k$, as follows:
\[
\theta_k(c) :=
(\rho \otimes \hat{\theta}^r_k) (\alpha_{n_k})^{-1} (\eins \otimes c),
\]
where $c \in \CC_{n_k}$, $\eins$ is the identity in $\CB(\CH)$,
and $\hat{\theta}^r_k := \theta^r_k \circ R_{n_k}$. 
\end{Proposition}

Remark: Note that even if $\theta^r_k$ is not a highly entangled state the corresponding $\theta_k$ may be highly entangled and so it is an interesting problem how to simplify this preparing sequence.

\begin{proof}
We have to show
$\lim_{k \to \infty} (\omega_\xi \otimes \theta_k) J_{n_k}(a) = \rho(a)$ for all $a \in \CB(\CH)$.

Let us denote the (one-dimensional) support projection of $\omega_\xi$ by $p_\xi$.
By assumption $\omega_\xi$ is $J^r$-preparable from $\rho$ with a preparing sequence $(\theta^r_k)$.
Hence, for any $\epsilon > 0$ we have for all $k$ big enough that
\[
(\rho \otimes \theta^r_k) (\alpha^{-1})_{n_k} (p_\xi \otimes \eins) =
(\rho \otimes \theta^r_k) J^r_{n_k} (p_\xi) > 1 - \epsilon\,.
\] 
Interchanging the positions in the tensor product by $R_n$ and noting $R_n \eins= \eins$ we also have

\begin{align*}
\label{form:einsminuseps}
(\rho \otimes \hat{\theta}^r_k) (\alpha_{n_k})^{-1} (p_\xi \otimes \eins) 
&=
(\rho \otimes \hat{\theta}^r_k) (\text{id}\otimes R_{n_k})(\alpha^{-1})_{n_k}(\text{id}\otimes R_{n_k})(p_\xi \otimes \eins) \\
&=
(\rho \otimes {\theta}^r_k)(\alpha^{-1})_{n_k}(p_\xi \otimes \eins) 
> 1 - \epsilon\,. \tag*{(*)}
\end{align*} 

The next step is based on the following elementary estimate.


\begin{Lemma} \label{lem:epsilon}
If $z$ is a bounded operator, $q$ a projection and $\varphi$ a state then
\[
	\varphi(q) > 1-\epsilon \quad {\rm implies} \quad | \varphi(z) - \varphi (qzq) | < 3 \, \sqrt{\epsilon}\, \|z\| \,.
\]
\end{Lemma}

The Lemma is obtained by writing
\[
z = qzq + (\eins-q)zq + qz(\eins-q) + (\eins-q)z(\eins-q) \,;
\]
now using the Cauchy-Schwarz inequality and $z^*z \le \|z\|^2 \eins$ and $\phi(q) > 1-\epsilon$ to get
\[
|\phi((\eins-q) z q) | \le \sqrt{\phi(\eins-q)} \, \sqrt{\phi(qz^*zq)}
< \sqrt{\epsilon}\, \|z\|
\] 
and similarly for the terms $qz(\eins-q)$ and $(\eins-q)z(\eins-q)$. 
The Lemma is proved.
\\

Suppose $z \in \CB(\CH) \otimes \CC_{n_k}$. 
With the formula given for $\theta_k$ in Proposition \ref{prop:reverse} and with $Q_{\omega_\xi}$ denoting the conditional expectation
determined by $Q_{\omega_\xi} (a \otimes c) = \omega_\xi(a)\,c$ we have
\[
(\omega_\xi \otimes \theta_k)(z) 
= \theta_k (Q_{\omega_\xi} (z))
=
(\rho \otimes \hat{\theta}^r_k) (\alpha_{n_k})^{-1} (\eins \otimes Q_{\omega_\xi} (z)).
\]
With Lemma \ref{lem:epsilon} for $q = p_\xi \otimes \eins$
we obtain by \ref{form:einsminuseps} for all $k$ big enough that
\[
| (\rho \otimes \hat{\theta}^r_k) (\alpha_{n_k})^{-1} (\eins \otimes Q_{\omega_\xi} (z)) - (\rho \otimes \hat{\theta}^r_k) (\alpha_{n_k})^{-1} ((p_\xi \otimes \eins)\; \eins \otimes Q_{\omega_\xi} (z)\; (p_\xi \otimes \eins)) |
< 3 \, \sqrt{\epsilon} \,\|z\|.
\]

To simplify the second term note that, because $p_\xi$ is one-dimensional, we have $p_\xi a p_\xi
= \omega_\xi(a) p_\xi$ for all $a \in \CB(\CH)$ and from that,
easily checked on elementary tensors,
\[
(p_\xi \otimes \eins) \;\eins \otimes Q_{\omega_\xi} (z) \;(p_\xi \otimes \eins) = (p_\xi \otimes \eins) \;z\; (p_\xi \otimes \eins)\,.
\]
So the inequality becomes
\[
| (\omega_\xi \otimes \theta_k)(z) - (\rho \otimes \hat{\theta}^r_k) (\alpha_{n_k})^{-1} (p_\xi \otimes \eins)\;  z \; (p_\xi \otimes \eins)) |
< 3 \, \sqrt{\epsilon} \,\|z\|.
\]
Again applying Lemma \ref{lem:epsilon} for $q = p_\xi \otimes \eins$ we also have for $k$ big enough that for all
$z \in \CB(\CH) \otimes \CC_{n_k}$
\[
| (\rho \otimes \hat{\theta}^r_k) (\alpha_{n_k})^{-1}(z)
- (\rho \otimes \hat{\theta}^r_k) (\alpha_{n_k})^{-1}((p_\xi \otimes \eins) z (p_\xi \otimes \eins))| < 3 \, \sqrt{\epsilon}\, \|z\|.
\]
Combining both inequalities we obtain
\[
| (\rho \otimes \hat{\theta}^r_k) (\alpha_{n_k})^{-1} (z) 
- (\omega_\xi \otimes \theta_k)(z) | 
< 6 \, \sqrt{\epsilon}\, \|z\|.
\]
Given any $a \in \CB(\CH)$ we can choose $z := J_{n_k}(a) =
\alpha_{n_k}(a \otimes \eins)$ and this becomes
\[
| (\rho \otimes \hat{\theta}^r_k) (a \otimes \eins) 
- (\omega_\xi \otimes \theta_k) J_{n_k}(a) | 
< 6 \, \sqrt{\epsilon}\, \|a\|.
\]
But $(\rho \otimes \hat{\theta}^r_k) (a \otimes \eins) = \rho(a)$
and with $k \to \infty$ we conclude that
\[
\lim_{k \to \infty} (\omega_\xi \otimes \theta_k) J_{n_k}(a) = \rho(a),
\]
as asserted.
\end{proof}

This gives us, after Theorem \ref{th:up}, another
sufficient condition for universal preparability of all normal states on $\CA = \CB(\CH)$.


\begin{Theorem} \label{th:reverse}
Let $J: \CB(\CH) \rightarrow \CB(\CH) \otimes \CC$ be a transition.
If there exists a vector state $\omega_\xi$, where $\xi$ is a unit vector in $\CH$, which is both universally
$J$-preparable and universally $J^r$-preparable then all normal states are universally $J$-preparable with compatible universally preparing sequences.
\end{Theorem}

\begin{proof}
By assumption $\omega_\xi$ is universally $J$-preparable.
Because $\omega_\xi$ is also universally $J^r$-preparable we can apply Proposition \ref{prop:reverse} for any normal state $\rho$ and find that all normal states on $\CB(\CH)$ are $J$-preparable from
$\omega_\xi$ with compatible preparing sequences. Concatenation of the two steps as in Corollary \ref{cor:composition} now yields the result.
\end{proof}

This theorem is motivated by  physics: If $\alpha$ describes some physical interaction between $\CA$ and $\CC$ then one would expect that preparing $\CC$ repeatedly in the ground state it should drive an arbitrary state on $\CA$ into the ground state $\omega_\xi$ on $\CA$, similarly for $\alpha^{-1}$. Thus from a physical point of view the assumption seems rather natural.

Indeed, we apply Theorem \ref{th:reverse} to a class of generalized micromaser interactions in Section \ref{section:example}. In this setting the idea of using a time reversal is a physically very plausible idea and it is discussed from this point of view also in \cite{WBKM00}. A similar idea in a framework of uploading and downloading quantum information can be found in \cite{BG07}.


\section{Tightness and Stationary States}
\label{section:stationary states}

In this section we develop the theory of tightness for normal states on $\CB(\CH)$ started in \cite{FR01} and prove a non-commutative version of Prokhorov's theorem. Then we collect some results about stationary states for unital positive maps which are relevant for our investigation. While many of these results are already well known there are others which we could not find in the literature and we think that in particular Theorem \ref{th:prok}, Proposition \ref{prop:order}, Corollary \ref{cor:faithful}, Proposition \ref{prop:orthogonal} and Proposition \ref{prop:absorbing-criteria} may be of independent interest and useful elsewhere too.


\begin{Definition} [\cite{FR01}] \label{def:tight}
A sequence $(\theta_n)^\infty_{n=0}$ of normal states on $\CB(\CH)$
is called tight if for all $\epsilon > 0$ there exists a finite dimensional projection $p \in \CB(\CH)$ such that for all $n$
\[
\theta_n(p) > 1- \epsilon\,.
\]
If $T: \CB(\CH) \rightarrow \CB(\CH)$ is a normal unital positive map
then we say that a normal state $\theta$ on $\CB(\CH)$ is tight with respect to $T$ if the sequence $(\theta \circ T^n)^\infty_{n=0}$
is tight. 
\end{Definition}

This is a generalization of a definition in classical probability which corresponds to the special case $\CA = L^\infty(\Omega, \Sigma, \mu)$, the essentially bounded functions on a probability space $(\Omega, \Sigma, \mu)$ (compare \cite{Sh96}, III.2, Definition 2).
Recall that we can identify normal functionals on $\CB(\CH)$ with trace class operators $\CT(\CH)$ and that the trace class operators are the dual of the compact operators $\CK(\CH)$, so we have the
$\sigma(\CT(\CH), \CK(\CH))$-topology as a weak$^*$ topology
on the predual $\CB(\CH)_*$.


\begin{Lemma} \label{lem:compact}
Let $(\theta_n)^\infty_{n=0}$ be a sequence of normal states on $\CB(\CH)$. If it converges weak$^*$ to a normal state $\theta$ then $(\theta_n)^\infty_{n=0}$ is also norm convergent to $\theta$. 
\end{Lemma}

\begin{proof}
Let $\epsilon >0$. Because $\theta$ is normal we can choose a finite dimensional projection $q \in \CK(\CH) \subset \CB(\CH)$ such that $\theta(q) > 1-\epsilon$. Because of the weak$^*$-convergence of
$\theta_n$ to $\theta$ there exists $n_0 \in \Nset$ so that for all $n > n_0$ also $\theta_n(q) > 1-\epsilon$. With Lemma \ref{lem:epsilon}
for $\theta$ and $\theta_n$, applied to $z$ with $\|z\| \le 1$, we get
\[
\| \theta - \theta_n \| \le \| \theta(q \cdot q) - \theta_n(q \cdot q) \|
+ 6 \sqrt{\epsilon}.
\]
But $\| \theta(q \cdot q) - \theta_n(q \cdot q) \| =
\| (\theta - \theta_n) (q \cdot q) \| \to 0$ for $n \to \infty$ because $q$ is finite dimensional.
\end{proof}

The following is a non-commutative version of a theorem of Prokhorov in classical probability theory, see for example
\cite{Sh96}, III.2, Theorem 1.


\begin{Theorem} \label{th:prok}
A sequence $(\theta_n)^\infty_{n=0}$ of normal states on $\CB(\CH)$ is tight if and only if it is relatively compact in the norm topology. 
\end{Theorem}

\begin{proof}
By Alaoglu's theorem, see \cite{KR86}, 1.6.5, the unit ball of $\CB(\CH)_*$ is weak$^*$ compact  and hence
there exists a subsequence $(\theta_{n_k})$ of $(\theta_n)^\infty_{n=0}$ ($\CH$ is separable) which converges weak$^*$ to a positive normal functional $\theta$. In particular if $q$ is a finite dimensional projection then $\theta_{n_k}(q) \to \theta(q)$ for $k \to \infty$ and tightness of $(\theta_n)^\infty_{n=0}$ implies that $\theta(\eins) = 1$, so
$\theta$ is a state. Now by Lemma \ref{lem:compact} the subsequence $(\theta_{n_k})$ converges to $\theta$ in the norm topology. This shows that if $(\theta_n)^\infty_{n=0}$ is tight then it is relatively compact in the norm topology. 

To get the converse, let $(q_k)^\infty_{k=0}$ be a sequence of finite dimensional projections in $\CB(\CH)$ which converges weak$^*$ to $\eins$. Suppose $(\theta_n)^\infty_{n=0}$ is not tight. Then there exists 
$\epsilon > 0$ such that for a subsequence $(\theta_{n_k})$ we have
$\theta_{n_k}(q_k) \le 1-\epsilon$. But if $(\theta_n)^\infty_{n=0}$
is relatively compact in the norm topology then a subsequence of $(\theta_{n_k})$ converges in the norm to a normal state and then the corresponding subsequence of $(\theta_{n_k}(q_k))$ converges to $1$. This cannot be the case.
\end{proof}


\begin{Definition}
Let $T$ be a normal unital positive map on a von Neumann algebra $\CA$. We say that a normal state $\phi$ on $\CA$ is stationary for $T$ if $\phi \circ T = \phi$. In this case we also write
$T: (\CA,\phi) \rightarrow (\CA,\phi)$.
\end{Definition}


\begin{Proposition} [\cite{FR01}] \label{prop:existence}
Let $T$ be a normal unital positive map on $\CB(\CH)$. There exists a stationary normal state $\phi$ for $T$ if and only if there is a normal state $\theta$ on $\CB(\CH)$ which is tight with respect to $T$. In fact, if $\theta$ is tight with respect to $T$ then the sequence $(\frac{1}{N} \sum^{N-1}_{n=0} \theta \circ T^n)^\infty_{N=0}$ converges to a stationary normal state (in the norm topology).
\end{Proposition}

\begin{proof}
Obviously a stationary normal state is tight with respect to $T$. Conversely, suppose that $\theta$ is any normal state which is tight with respect to $T$. Then $( \theta \circ T^n)^\infty_{n=0}$ is a tight sequence and hence 
$(\frac{1}{N} \sum^{N-1}_{n=0} \theta \circ T^n)^\infty_{N=0}$
is also a tight sequence. We infer from Theorem \ref{th:prok} and its proof that there exists a normal state $\phi$ which is an accumulation point of this sequence. Now it follows from standard results in ergodic theory, see for example \cite{Kr85}, Chapter 2, Theorem 1.1 applied to the preadjoint $T_*$ of $T$, that we actually have
\[
\phi = \lim_{N \to \infty} 
\frac{1}{N} \sum^{N-1}_{n=0} \theta \circ T^n
\]
(in the norm topology). Further
\[
\| \phi \circ T - \phi \| \le \lim_{N \to \infty} \| \frac{1}{N} (\theta - \theta \circ T^N) \| = 0,
\]
so $\phi$ is stationary.
\end{proof}


\begin{Proposition} \label{prop:order}
If $T$ is a normal unital positive map on $\CB(\CH)$ then there
exists a projection $p_t \in \CB(\CH)$ such that a normal
state $\theta$ is tight with respect to $T$ if and only if $\supp\,\theta \le p_t$.
\end{Proposition}

\begin{proof}
We make use of the following result in \cite{Da76}, Chapter 4, Lemma 3.2: For any norm closed order ideal $\mathcal{O}$ in the set of all normal positive linear functionals on $\CB(\CH)$ there exists a projection $p \in \CB(\CH)$ so that a normal positive linear functional $\tau$ on $\CB(\CH)$ is in $\mathcal{O}$ if and only if
$\supp\,\tau \le p$. Recall that by definition $\mathcal{O}$ is an order ideal if it is a cone such that $0 \le \rho \le \tau \in \mathcal{O}$ implies $\rho \in \mathcal{O}$.

Hence Proposition \ref{prop:order} is proved if we can show that
the set $\mathcal{O}$ of all functionals $c \theta$ with $0 \le c \in \Rset$ and $\theta$ tight with respect to T is a norm closed order ideal. 

$\mathcal{O}$ is a cone. In fact, if $0 \not= \theta_1, \theta_2 \in \mathcal{O}$ then for any $\epsilon >0$ there exist finite dimensional projections $p_1, p_2$ such that $\theta_i \circ T^n (p_i) > 
\|\theta_i\| \, (1-\epsilon)$ for $i=1,2$ and all $n$. The supremum $p := p_1 \vee p_2$
is also a finite dimensional projection and we get for all $\lambda_1, \lambda_2 \ge 0$ that
\[
(\lambda_1 \theta_1 + \lambda_2 \theta_2) \circ T^n (p) 
> (\lambda_1 \|\theta_1\| + \lambda_2 \|\theta_2\|) \,(1-\epsilon)
= \| \lambda_1 \theta_1 + \lambda_2 \theta_2 \| \, (1-\epsilon),
\]
hence $\lambda_1 \theta_1 + \lambda_2 \theta_2 \in \mathcal{O}$.

$\mathcal{O}$ is norm closed. In fact, suppose that $\theta$ is a normal state such that for all $\delta >0$ there exist $\theta_\delta 
\in \mathcal{O}$ such that $\| \theta - \theta_\delta \| < \delta$.
Note that $\| \theta_\delta \| > 1 - \delta$. For any $\epsilon >0$
there exists a finite dimensional projection $p$ so that
$\theta_\delta \circ T^n (p) > \| \theta_\delta \|\,(1-\epsilon)$ for all $n$. Then 
\[
\theta \circ T^n (p) = \theta_\delta \circ T^n(p) + (\theta-\theta_\delta) \circ T^n (p)
> (1-\delta) \,(1-\epsilon) - \delta
\]
for all $n$. Hence $ \theta \in \mathcal{O}$. A similar argument applies to any multiple $c \theta$ with $c >0$. 

Finally suppose that $0 \le \rho \le \theta \in \mathcal{O}$. We may assume that $\rho \not=0$ and $\theta$ is a state. Because $\theta$ is tight, for all $\epsilon >0$ there exists a finite dimensional projection $q$ such that
$\theta \circ T^n (q) > 1 - \| \rho \|\, \epsilon$ for all $n$. Then
\[
1 - \| \rho \|\, \epsilon < \theta \circ T^n (q) =
\|\rho\| \,\big( \frac{1}{\|\rho\|} \rho \circ T^n(q) \big)
+ \| \theta-\rho \| \,\big( \frac{1}{\|\theta-\rho\|} (\theta-\rho) \circ T^n(q)\big).
\]
Because $\|\theta-\rho\| = (\theta-\rho)(\eins) = 1 - \|\rho\|$ we get for all $n$ that
\[
\frac{1}{\|\rho\|} \rho \circ T^n(q) > \frac{1}{\|\rho\|} \,\big(
1 - \| \rho \| \epsilon - (1 - \|\rho\|) \big) = 1 - \epsilon,
\]
hence $\rho \in \mathcal{O}$.
\end{proof}


\begin{Corollary} \label{cor:faithful}
If there exists a faithful normal state which is tight with respect to a normal unital positive map $T$ on $\CB(\CH)$ then all normal states are tight with respect to $T$.
\end{Corollary}


\begin{Definition}[\cite{FR02}] \label{def:harmonic}
Let $T$ be a normal unital positive map on a von Neumann algebra $\CA$. A positive element $a \in \CA$ is called
\begin{eqnarray*}
\text{superharmonic} & \text{if} & a \ge T(a), \\
\text{subharmonic} & \text{if} & a \le T(a), \\
\text{harmonic} & \text{if} & a = T(a). 
\end{eqnarray*}
\end{Definition}


\begin{Proposition} \label{prop:harmonic}
Let $T$ be a normal unital positive map on a von Neumann algebra $\CA$ and let $p \in \CA$ be a projection.
\begin{itemize}
\item[(1)]
$p$ is superharmonic if and only if 
$T(pap) = p\, T(pap)\, p$ for all $a \in \CA$.
\item[(2)]
$p$ is subharmonic if and only if
$p\, T(a)\, p = p\, T(pap)\, p$ for all $a \in \CA$.
\item[(3)]
$p$ is harmonic if and only if $p\, T(a)\, p = T(pap)$ for all $a \in \CA$.
\item[(4)]
The support $\supp\,\phi$ of a stationary normal state $\phi$ is subharmonic.
\item[(5)]
If there exists a stationary faithful state for $T$ then all positive superharmonic or subharmonic elements are harmonic.
\end{itemize}
\end{Proposition}

\begin{proof}
The nontrivial parts can be derived from Lemma 2 in \cite{Lu95},
see also Section 2 of \cite{FR02} or \cite{GK12}. For convenience we include a short proof which is valid for our setting.

(1) Inserting $a=\eins$ into the right hand side yields $T(p) = p\,T(p)\,p$, hence $T(p) \le \|T(p)\|\,p \le p$ and $p$ is superharmonic.
Conversely, if $p$ is superharmonic then for $a \ge 0$
\[
0 \le T(pap) \le \|a\| \,T(p) \le \|a\| \,p
\]
which implies $T(pap) = p\, T(pap)\, p$. Because all $a \in \CA$ are linear combinations of positive elements this is valid for all $a$.

(2) Inserting $a=\eins$ yields $p = p\,T(p)\,p$, hence 
\[
0 = p\, T(\eins-p) \,p \ge \big( T(\eins-p)p \big)^* \big( T(\eins-p)p \big) \ge 0,
\]
which implies $T(\eins-p)\,p=0$. Here we have used the Kadison-Schwarz inequality
for the selfadjoint projection $\eins-p$, 
see
\cite{St13}, Theorem 1.3.1(ii).  
We conclude that
\[
T(\eins-p) = (\eins-p)\, T(\eins-p)\, (\eins-p) \le \| T(\eins-p) \|\, (\eins-p) \le \eins-p,
\]
hence $p \le T(p)$ and $p$ is subharmonic. 

Conversely, if $p$ is subharmonic and $\theta$ is any normal state then we define
the normal positive functional $\theta_p(\cdot) := \theta(p \cdot p)$. Note that 
\[
T_* \theta_p(\eins-p) = \theta_p (T(\eins-p)) \le \theta_p (\eins-p) =0,
\]
so $\supp(T_* \theta_p) \le p$ and $T_* \theta_p = (T_* \theta_p)_p$. Inserting the definitions gives
$\theta (p\, T(a)\, p) = \theta (p\, T(pap)\, p)$ for all $a \in \CA$.
Because this is true for all $\theta$ we get the stated result.

(3) follows by combining (1) and (2). 

(4) Putting $p_\varphi:=\supp \varphi$ and $p_\varphi^\perp:=\eins - p_\varphi$, we have $0 \le T(p_\varphi^\perp) \le \eins$ and
$\varphi ( T(p_\varphi^\perp)) = \varphi (p_\varphi^\perp) = 0$,
hence $T(p_\varphi^\perp) \le p_\varphi^\perp$ and thus
$p_\varphi \le T(p_\varphi)$.

(5) If $ 0 \le a \in \CA$ is superharmonic and $\varphi$ is a stationary faithful state on $\CA$ then $0 \le \varphi (a - T(a))
= \varphi(a)-\varphi(a) = 0$ and hence $a = T(a)$. A similar argument applies to subharmonic elements. 
\end{proof}


\begin{Proposition} \label{prop:orthogonal}
Let $T$ be a normal unital positive map on a von Neumann algebra $\CA$ and $\varphi, \psi$ stationary normal states such that $\supp\psi$ is not dominated by $\supp\varphi$. Then there exists a stationary normal state $\varphi_\perp$ such that
$\supp\varphi \perp \supp\varphi_\perp$.
\end{Proposition}

\begin{proof}
Let $p := \supp\theta$ with $\theta := \frac{1}{2} (\varphi + \psi)$. Then we infer from the assumptions about the supports that $\supp\varphi \le p$ and $q := p- \supp\varphi \not=0$. Now we define
\[
T_p: p \CA p \rightarrow p \CA p, \quad y \mapsto p\, T(y) p,
\]
which is a normal unital (on $p \CA p$) positive map with
a stationary faithful normal state $\theta |_{p \CA p}$. Note that, by Proposition \ref{prop:harmonic}(4), $p$ and $\supp\,\varphi$ are subharmonic and hence $q$ is superharmonic for both $T$ and $T_p$. By Proposition \ref{prop:harmonic}(5) these projections are harmonic for $T_p$. 

Claim: The normal positive functional $\theta_q(\cdot) := \theta( q \cdot q)$ is stationary for $T$. This is shown by the following computation which involves Proposition \ref{prop:harmonic}(1) for $T$ and Proposition \ref{prop:harmonic}(3) for $T_p$. Let $a \in \CA$. Then
\[
\theta_q \circ T(a) = \theta (q \,T(a) q) = \theta (qp \,T(a) pq)
= \theta (qp \,T(pap) pq) 
\]
\[
= \theta (q \,T_p(pap) q)
= \theta (T_p(qpapq)) = \theta (T_p(qaq)) = \theta(qaq) = \theta_q(a)\,.
\]
Because $q \perp \supp\varphi$ we can now define $\varphi_\perp := \frac{\theta_q}{\|\theta_q\|}$.
\end{proof}
For completely positive $T$ this result follows also from \cite{GK12}, Theorem 7.1.

Finally we include a few results about absorbing states which will be needed later. The following proposition is well known.


\begin{Proposition} [\cite{Go04}, A.5.2] \label{prop:absorbing}
Let $T$ be a normal unital positive map on $\CB(\CH)$ and $\xi \in \CH$ a unit vector.
Let $\omega_\xi$ be the corresponding vector state with (one-dimensional) support $p_\xi$.

The following assertions are equivalent:
\begin{itemize}
\item[(a1)]
The state $\omega_\xi$ is absorbing, i.e., for all normal states $\theta$ on $\CB(\CH)$ and all $a \in \CB(\CH)$
\[
\lim_{n\to\infty} \theta \circ T^n(a) = \omega_\xi(a)\,.
\]
\item[(a2)]
For all normal states $\theta$ and all $a \in \CB(\CH)$ in the norm topology
\[
\lim_{n\to\infty} \theta \circ T^n = \omega_\xi\,.
\]
\item[(b)]
$T$ is ergodic, i.e., there are only trivial fixed points (namely $\Cset \eins$), \\
and $\omega_\xi$ is stationary.
\item[(c1)]
$\lim_{n\to\infty} T^n(p_\xi) = \eins$ in the strong operator topology
\item[(c2)]
$\lim_{n\to\infty} T^n(p_\xi) = \eins$ in the weak$^*$ topology
\end{itemize}
If we replace $\omega_\xi$ by an arbitrary normal state $\varphi$ (not necessarily a vector state)
then we still have
\[
(a1) \Leftrightarrow (a2) \Rightarrow (b) \Rightarrow (c1) \Leftrightarrow (c2).
\]
\end{Proposition}


\begin{Proposition} \label{prop:absorbing-criteria}
Let $T$ be a normal unital positive map on $\CB(\CH)$ such that all normal states are tight with respect to $T$. 
\begin{itemize}
\item[(i)]
If there is a stationary vector state which is the only stationary normal state for $T$ then it is absorbing.
\item[(ii)]
Suppose the vector state $\omega_\xi$ is stationary and let $p_\xi$
be its support. Then 
$\omega_\xi$ is absorbing if and only if for every normal state $\theta$ on $\CB(\CH)$ there exists $n \in \Nset$ such that
$\theta \circ T^n(p_\xi) \not=0$.
\end{itemize}
\end{Proposition}

\begin{proof}
(i) By Proposition \ref{prop:existence} it follows that for any normal state $\theta$ the sequence $(\frac{1}{N} \sum^{N-1}_{n=0} \theta \circ T^n)$ converges in the norm to a stationary normal state. This must be the given stationary vector state $\omega_\xi$ because by assumption there is no other stationary normal state. 
Now we want to use Proposition \ref{prop:absorbing}(c1):
If $p_\xi$ is the support of this vector state we have
\[
\lim_{N \to \infty} \frac{1}{N} \sum^{N-1}_{n=0} \theta \circ T^n (p_\xi) = \omega_\xi(p_\xi) = 1\,.
\]
On the other hand $p_\xi$ is subharmonic by Proposition \ref{prop:harmonic}(4) which implies that $a := \lim_{n \to \infty} T^n(p_\xi)$ exists in the strong operator topology and $T(a)=a$. Then we can also write $a$ as a Cesaro limit in the strong operator topology, namely 
$a = \lim_{N \to \infty} \frac{1}{N} \sum^{N-1}_{n=0} T^n (p_\xi)$. Comparison with the limit above gives $\theta(a)=1$ for all
normal states $\theta$, hence $a=1$. 

(ii) If $\omega_\xi$ is absorbing then the given criterion is satisfied because $\theta \circ T^n(p_\xi) \to \omega_\xi(p_\xi) = 1$ for $n \to \infty$ for all normal states $\theta$. For the converse assume that
$\omega_\xi$ is not absorbing. Then it follows from (i) that
$\omega_\xi$ is not the only stationary normal state and by Proposition \ref{prop:orthogonal} there exists a stationary normal state $\varphi$ with $supp\,\varphi$ orthogonal to $p_\xi$. But then $\varphi \circ T^n(p_\xi) = \varphi(p_\xi) = 0$ for all $n$. 
\end{proof}


\section{Transitions and Stationary States}
\label{section:stationary markov}

We first review a basic notion of a 
(discrete-time, one-sided) non-commutative stationary Markov chain which matches the one used in \cite{GKL06}. It will turn out in the next section that we can make use of it to deepen our understanding of universal preparability. We concentrate on the transition specifying the Markov chain and do not say much about the chain itself.
A broader discussion and motivation of non-commutative stationary Markov chains can be found in \cite{GKL06} and in \cite{Ku06}. 

Suppose that a transition $J: \CA \rightarrow \CA \otimes \CC$ is given and let us fix a normal state $\psi$ on $\CC$. 
Then the associated {\it transition operator}
\[
T_\psi := P_\psi J\,,
\]
where $P_\psi: \CA \otimes \CC \rightarrow \CA$ is the conditional expectation determined by $P_\psi (x \otimes y) = x \,\psi(y)$, is a normal unital completely positive map on $\CA$. 

We can put copies $\psi_{(j)}$ of $\psi$ on copies $\CC_{(j)}$ of $\CC$ and define a normal product state
$\psi_n := \psi_{(1)} \otimes \ldots \otimes \psi_{(n)}$ on $\CC_n := \bigotimes^n_{j=1} \CC_{(j)}$. Then we can check that
\[
(T_\psi)^n = P_{\psi_n} J_n \,,
\]

this reflects the Markovian character of the time evolution (cf. \cite{Ku85}, 2.2.7).

For a normal state $\varphi$ on $\CA$ it is easy to check that the following assertions are equivalent:

\begin{itemize}
\item $\varphi \circ T_\psi = \varphi$, we write $T_\psi: (\CA, \varphi) \rightarrow (\CA, \varphi)$
\item $(\varphi \otimes \psi) J = \varphi$, we write $J: (\CA,\varphi) \rightarrow (\CA \otimes \CC, \varphi \otimes \psi)$
\item $(\varphi \otimes \psi_n) J_n = \varphi$ for all $n \in \Nset$
\end{itemize}

The first assertion says that $\varphi$ is a stationary normal state for $T_\psi$. Note that while the existence of stationary states follows from fixed point theorems the existence of stationary normal states, as requested here, is a non-trivial assumption if $\CA$ is infinite dimensional.
The last assertion can be interpreted by saying that $J: (\CA,\varphi) \rightarrow (\CA \otimes \CC, \varphi \otimes \psi)$ specifies a non-commutative stationary Markov chain.
If only the transition $J$ is given then in order to construct a stationary Markov chain we can choose any normal state $\psi$ on $\CC$ and then check if we can find a corresponding stationary normal state $\varphi$ for $T_\psi$ on $\CA$. In the following we are particularly interested in faithful states.

In the setting of our main results, see Theorem \ref{th:main}, we always have faithfulness of the stationary state $\varphi$. This is related to the following concepts which we also quickly review here, for proofs, further motivation and discussion we refer to \cite{GKL06}. Our notation is slightly different from \cite{GKL06} because here we avoid the use of infinite tensor products. 


\begin{Definition} \label{def:irrep}
A transition $J: \CA \rightarrow \CA \otimes \CC$ is called irreducible
if $p=0$ and $p=\eins$ are the only projections $p \in \CA$ which
satisfy $J(p) \le p \otimes \eins$.

A positive unital map $T$ is called irreducible if $0$ and $\eins$ are the only subharmonic projections (or, equivalently by replacing $p$ by $\eins-p$, if $0$ and $\eins$ are the only superharmonic projections). It is called ergodic, if its fixed space is one-dimensional and hence given by $\Cset\eins$.

\end{Definition}


\begin{Proposition}[\cite{GKL06}, Section 1]
\label{prop:irreducible-JT}
Let $J:\CA \to \CA \otimes \CC$ be a transition and
$\psi$ on $\CC$ a faithful normal state. Suppose further
that the transition operator $T_\psi$ has a stationary normal
state $\varphi$. Then $J$ is irreducible if and only if $T_\psi$ is irreducible.
\end{Proposition}

Recall that in a $C^*$-algebra a conditional expectation is defined as an idempotent linear map of norm one onto a $C^*$-subalgebra, see 
\cite{Tak79}, III.3.3.


\begin{Theorem}[\cite{KN79}] \label{th:KN}
Let $T$ be a normal unital completely positive map on a von Neumann algebra $\CA$ with a stationary faithful normal state $\varphi$. Then the set $\CF(T)$ of all fixed points is a von Neumann subalgebra of $\CA$ and there exists a unique normal conditional expectation $P$ onto $\CF(T)$ which preserves the state $\varphi$, i.\,e., $\varphi \circ P = \varphi$. The set of all stationary normal states of $T$ is equal to
$\{ \theta \circ P \colon \theta \; \text{is a normal state on} \; \CF(T) \}$.
\end{Theorem}


\begin{Proposition} \label{prop:irreducible-ergodic}
Suppose $T$ is a normal unital completely positive map on a von Neumann algebra $\CA$ with a stationary normal state $\varphi$.
Then the following assertions are equivalent:
\begin{itemize}
\item[(a)]
$T$ is irreducible.
\item[(b)]
$T$ is ergodic and $\varphi$ is faithful.
\end{itemize}
If the assertions are satisfied then $\varphi$ is the unique stationary normal state for $T$.
\end{Proposition}

\begin{proof}
Assume that $T$ is irreducible. By Proposition \ref{prop:harmonic}(4) $\supp\varphi$ is subharmonic, so by irreducibility $\supp\varphi = \eins$ and $\varphi$ is faithful. It follows that the fixed points form a von Neumann subalgebra $\CF(T)$, by Theorem \ref{th:KN}. By irreducibility any projection in this subalgebra must be $0$ or $\eins$, hence $T$ is ergodic. 

Conversely assume that $T$ is not irreducible and that $\varphi$ is faithful. We choose a non-trivial superharmonic projection $p$. Then the sequence $(T^n p)$ converges weak$^*$ to a fixed point $a \ge 0$. Because $a \le p$ we have $a \not= c \eins$ for $0 \not= c \in \Cset$. Because $\varphi$ is faithful and $\varphi(a) = \varphi(p) \not= 0$
we also have $a \not=0$. Hence $T$ is not ergodic.

Because the fixed point algebra $\CF(T)$ is one dimensional it follows immediately from the characterization of the set of stationary normal
states in Theorem \ref{th:KN} that $\varphi$ is unique.
\end{proof}

Stationary states make it possible to introduce further tools into the investigation of transitions. Among them are the extended transition operators. See \cite{Go04} for detailed background about these. For the convenience of the reader we review their definition including some further context and we add a few additional directions to this topic, based on Section \ref{section:stationary states} and applied later in Theorem \ref{th:main}.

If we have a non-commutative stationary Markov process specified by
$J: (\CA,\varphi) \rightarrow (\CA \otimes \CC, \varphi \otimes \psi)$,
where $\varphi$ and $\psi$ are faithful normal states, then there is further structure available from the GNS-construction applied to these states. We use notation as follows. There is an inner product
$\langle a,  b \rangle_\varphi := \varphi (a^*b)$ on $\CA$ and $\CA$ becomes a dense subspace in the GNS-Hilbert space $\CH_\varphi$ with respect to the corresponding norm $\|\cdot\|_\varphi$. To simplify notation we also identify $\CA$ with its GNS-representation on $\CH_\varphi$. In this sense $\CA \subset \CB(\CH_\varphi)$. If we want to make explicit that we consider an element $a \in \CA$ as an element of $\CH_\varphi$ then we write $a \eins$ and think of $\eins$ as the cyclic vector. But for the norm $\langle a \eins, a \eins \rangle_{\varphi}^{\frac{1}{2}} = \varphi(a^* a)^{\frac{1}{2}}$ we often further simplify notation and just write $\|a\|_\phi$ instead of $\|a \eins\|_\phi$.
Similar conventions apply to the other algebras, for example the algebra $\CA \otimes \CC_n$ gives rise to an inner product
$\langle \cdot, \cdot \rangle_{\varphi \otimes \psi_n}$ for the GNS-Hilbert space $\CH_\varphi \otimes \CH_{\psi_n}$, etc.


\begin{Proposition} \label{prop:irr-dual}
Let $T: \CA \rightarrow \CA$ be a normal unital completely positive map with a stationary normal state $\varphi$ and let $\CA'$ be the commutant of $\CA$ in $\CB(\CH_\varphi)$.
Further let $T': \CA' \rightarrow \CA'$ be the dual operator in the sense that we have for all $a \in \CA$ and $a' \in \CA'$ 
\[
\langle \eins, T(a)\, a' \,\eins \rangle_\varphi = \langle \eins, a\, T'(a') \,\eins \rangle_\varphi\,.
\]
Then $T$ is irreducible if and only if $T'$ is irreducible.
\end{Proposition}

\begin{proof}
The dual operator $T': \CA' \rightarrow \CA'$ is also a normal unital completely positive map, see for example \cite{GK82} or \cite{Go04}, 1.5.1.
Suppose $T$ is irreducible. Then by Proposition \ref{prop:irreducible-ergodic} the stationary normal state $\varphi$ is faithful and
$\eins$ becomes a cyclic and separating vector for $\CA$. But then also the restriction of the vector state $\omega_\eins = \langle \eins, \cdot \; \eins \rangle_\varphi $ to $\CA'$ is faithful and hence the fixed point space of $T'$ is a von Neumann subalgebra $\CF(T')$ in $\CA'$, see Theorem \ref{th:KN}. 
Let $p'$ be any projection in $\CF(T')$. Then for $a \in \CA$
\[
\langle p' \eins, T(a)\, p' \eins \rangle_\varphi 
= \langle \eins, a\, T'(p') \eins \rangle_\varphi 
= \langle \eins, a p' \eins \rangle_\varphi
= \langle p' \eins, a p' \eins \rangle_\varphi,
\]
hence $\omega_{p' \eins}(\cdot) = \langle p' \eins, \cdot \,p' \eins \rangle_\varphi$ restricted to $\CA$ is a stationary normal positive functional for $T$. But because $T$ is irreducible we get from Proposition \ref{prop:irreducible-ergodic}
that $\varphi$ is the only stationary normal state, hence $\omega_{p' \eins}$ restricted to $\CA$ must be a scalar multiple of $\varphi$. In other words there exists $\lambda \ge 0$ so that for all
$a \in \CA$ we have $\langle a \eins, p' \eins \rangle_\varphi = \lambda \langle a \eins, \eins \rangle_\varphi$, so $p' = \lambda \eins$. This shows that $T'$ is ergodic and it has a stationary faithful normal state. We conclude by Proposition \ref{prop:irreducible-ergodic} that $T'$ is irreducible. The other direction is now also clear because $(T')'=T$. 
\end{proof}

For a non-commutative stationary Markov process specified by
$J: (\CA,\varphi) \rightarrow (\CA \otimes \CC, \varphi \otimes \psi)$,
with $\varphi$ and $\psi$ faithful normal states, it follows from stationarity 
that $J$ can be extended to an isometry $v: \CH_\varphi \rightarrow \CH_\varphi \otimes \CH_\psi$. The dual extended transition operator $Z'$ is the normal unital completely positive map on $\CB(\CH_\varphi)$ defined by the following Stinespring representation:
\[
Z'(x) := v^* \; x \otimes \eins \; v\,.
\]


\begin{Proposition}[\cite{GKL06}, Section 4] \label{prop:Z-review}
Properties of $Z'$.
\begin{itemize}
\item[(1)]
The vector state $\omega_\eins = \langle \eins, \cdot \; \eins \rangle$ for $\eins \in \CA \eins \subset \CH_\varphi$ is stationary for $Z'$.
\item[(2)]
Let $\CA'$ be the commutant of $\CA$ in $\CB(\CH_\varphi)$. Then
$Z'(\CA') \subset \CA'$ and $Z'|_{\CA'} = T'$.
\item[(3)] Suppose there exists a $(\varphi \otimes \psi)$-preserving conditional expectation from $\CA \otimes \CC$ onto $J(\CA)$. Then
$Z'(\CA) \subset \CA$ and $Z' |_\CA = T^+$, where
\[
T^+: \CA \rightarrow \CA, \quad \langle T^+(a) \eins, b \eins \rangle_\varphi =
\langle a \eins, T(b) \eins \rangle_\varphi \quad (a,b \in \CA).
\]
\end{itemize}
\end{Proposition}
The identity above means that $T^+$ is the adjoint of $T$ with respect to the state $\varphi$ and may be rewritten as $\varphi(T^+(a)^*\,b)=\varphi(a^*\,T(b))$ (cf. also \cite{GK82}, section 3).

For later use in the proof of Theorem \ref{th:main} below we include the following observation. 


\begin{Lemma} \label{lem:Z-stationary}
If $J$ is an irreducible transition such that 
$J: (\CA,\varphi) \rightarrow (\CA \otimes \CC, \varphi \otimes \psi)$,
with $\varphi$ and $\psi$ faithful normal states, then every stationary normal state $\eta$ for the corresponding dual extended transition operator $Z'$ has the form
\[
\eta = \sum^\infty_{j=1} \omega_{a_j \eins} 
\quad \text{with}\;\; a_j \in \CA \;\;\text{for all}\; j.
\]
\end{Lemma} 

\begin{proof}
Let $\eta$ be a normal state on $\CB(\CH_\varphi)$ which is stationary for $Z'$. Like any normal state on $\CB(\CH_\varphi)$ we can write $\eta$ in the form $\eta = \sum^\infty_{j=1} \omega_{\xi_j}$ with $\xi_j \in \CH_\varphi$, see \cite{KR86}, Theorem 7.1.12. We have to show that under the given assumptions we can choose $\xi_j = a_j \eins$ with $a_j \in \CA$, for all $j$.

It follows from Proposition \ref{prop:irreducible-JT} and Proposition \ref{prop:irr-dual} that with $J$ also $T := T_\psi$ and $T'$ are irreducible. If $\eta$ is stationary for $Z'$ then because $Z' |_{\CA'} = T'$ (from Proposition \ref{prop:Z-review}(2)) the restriction of $\eta$ to $\CA'$ is stationary for $T'$. Because $T'$ is irreducible we get from Proposition \ref{prop:irreducible-ergodic} that there cannot be more than one stationary normal state, hence the restrictions of $\eta$ and
of $\omega_\eins$ to $\CA'$ coincide. In particular
\[
\omega_{\xi_j} |_{\CA'} \le \omega_\eins |_{\CA'} 
\quad \text{for all}\; j .
\]

We now want to define $a_j$ as an operator on $\CH_\varphi$, for all $j$. Note that $\CA \eins$ is a dense subspace of $\CH_\varphi$,
so there is a sequence $(a^{(n)}_j)_{n \in \Nset} \subset \CA$ such that
$\| (a^{(n)}_j \eins - \xi_j \|_\varphi \to 0$ for $n \to \infty$. We define $a_j$
first on the dense subspace $\CA' \eins \subset \CH_\varphi$ by
\[
a_j (b' \eins) := \lim_{n \to \infty} a^{(n)}_j b' \eins
= \lim_{n \to \infty} b' a^{(n)}_j \eins = b' \xi_j,
\]
where $b' \in \CA'$. We compute
\[
\| a_j (b' \eins) \|^2_\varphi = \langle a_j (b' \eins), a_j (b' \eins) \rangle_\varphi = \langle b' \xi_j, b' \xi_j \rangle_\varphi
= \omega_{\xi_j} (b'^*b') 
\le \omega_\eins (b'^*b') = \|b' \eins \|^2_\varphi\,,
\]
from which we conclude that $a_j$ is bounded and can be extended to an operator $a_j \in \CB(\CH_\varphi)$. We can easily check on the dense subspace $\CA' \eins \subset \CH_\varphi$ that it commutes with $\CA'$. Indeed, if $b', c' \in \CA'$ then
\[
a_j c' (b' \eins) = \lim_{n \to \infty} a^{(n)}_j c' b' \eins
= \lim_{n \to \infty} c' a^{(n)}_j b' \eins = c' a_j (b' \eins).
\]
Hence $a_j \in \CA'' = \CA$ and, with $b'=\eins$, we have
$a_j \eins = \xi_j$, as claimed.
\end{proof}

We now use the dual extended transition operator $Z'$ to define a concept of tightness for a transition $J$.


\begin{Definition} \label{def:J-tight}
Let $J: (\CA,\varphi) \rightarrow (\CA \otimes \CC, \varphi \otimes \psi)$ be a transition for faithful normal states $\varphi$ on $\CA$ and $\psi$ on $\CC$. 
If all normal states on $\CB(\CH_\varphi)$ are tight for the corresponding dual extended transition operator $Z'$
then we say that $J: (\CA,\varphi) \rightarrow (\CA \otimes \CC, \varphi \otimes \psi)$ is tight.
\end{Definition}

We emphasize that tightness is a rather weak property for a transition $J$ which in practice is satisfied in many cases of interest. This is illustrated by the following proposition.


\begin{Proposition} \label{prop:J-tight}
Let $J: \CA \rightarrow \CA \otimes \CC$ be a transition.

If one of the following conditions (1) or (2) is satisfied:
\begin{itemize}
\item[(1)]
$\CA = \CB(\CH)$ and $\CC = \CB(\CK)$ with $\CH$ and $\CK$ finite dimensional Hilbert spaces,
\item[(2)]
$\CA$ is finite dimensional and $J$ is irreducible,
\end{itemize}
then there exist faithful normal states $\varphi$ and $\psi$
such that $J: (\CA,\varphi) \rightarrow (\CA \otimes \CC, \varphi \otimes \psi)$ and for all such choices the latter is tight.
\begin{itemize}
\item[(3)]
Let $\CA = \CB(\CH)$ with a Hilbert space $\CH$ and
$J: (\CB(\CH),\varphi) \rightarrow (\CB(\CH) \otimes \CC, \varphi \otimes \psi)$
with faithful normal states $\varphi$ on $\CB(\CH)$ and $\psi$ on $\CC$ and suppose that there exists a $(\varphi \otimes \psi)$-preserving conditional expectation from $\CB(\CH) \otimes \CC$ onto $J(\CB(\CH))$. Then $J: (\CB(\CH),\varphi) \rightarrow (\CB(\CH) \otimes \CC, \varphi \otimes \psi)$ is tight.
\item[(4)]
Let $\CA = \CB(\CH)$ with a Hilbert space $\CH$ and the transition $J$ be induced by a coupling automorphism $\alpha$ of $\CB(\CH) \otimes \CC$ such that $J(a) = \alpha(a \otimes \eins)$ for all $a \in \CB(\CH)$. Suppose further that $\varphi$ on $\CB(\CH)$ and $\psi$ on $\CC$ are faithful normal states such that $\varphi \otimes \psi$ is stationary for the automorphism $\alpha$. 
Then $J: (\CB(\CH),\varphi) \rightarrow (\CB(\CH) \otimes \CC, \varphi \otimes \psi)$ is tight.
\end{itemize}
\end{Proposition}

\begin{proof}
In both (1) and (2) we have tightness of all normal states for $Z'$ because $\CA$ and hence also $\CB(\CH_\varphi)$ are finite dimensional. Further in (1) we can choose the (faithful) tracial states for $\varphi$ and $\psi$. In (2) choose any faithful normal state $\psi$ on $\CC$ and then any $T_\psi$-stationary state $\varphi$ on $\CA$, which exists by a fixed point argument and is normal because $\CA$ is finite dimensional. Because $J$ is irreducible it follows from Proposition \ref{prop:irreducible-JT} that also $T_\psi$ is irreducible and hence, by Proposition \ref{prop:irreducible-ergodic}, $\varphi$ is faithful.

In (3) we have to prove the tightness of all normal states for $Z'$.
Because $\CA = \CB(\CH)$, write $\CB(\CH_\varphi) = \CA \otimes \CA'$ where $\CA' = \CB(\CH')$ with a Hilbert space $\CH'$. Further we take from Proposition \ref{prop:Z-review}
that $Z'$ restricted to $\CA' \simeq \eins \otimes \CA'$ is $T'$ and
$Z'$ restricted to $\CA \simeq \CA \otimes \eins$ is $T^+$ (here the existence of the conditional expectation is used).
Because $\eins$ is a cyclic and separating vector for $\CA$ it follows that the restriction $\varphi'$ of $\omega_\eins$ to $\CA'$ is a stationary faithful normal state for $T'$ and the restriction $\varphi$ of $\omega_\eins$ to $\CA$ is a stationary faithful normal state for $T^+$. We conclude 
from Corollary \ref{cor:faithful} that all normal states on $\CA'$ are tight for $T'$ and all normal states on $\CA$ are tight for $T^+$.
Now let $\theta$ be any normal state on $\CB(\CH_\varphi)$. 
Then for any $\epsilon >0$ there exist finite dimensional projections $p' \in \CA'$ and $p \in \CA$ such that for all $n$
\[
\theta \circ (T')^n (p')      \; >\;    1 - \frac{\epsilon}{2}\,, \quad
\theta \circ (T^+)^n (p)   \;> \;   1 - \frac{\epsilon}{2}.
\]
For the finite dimensional projection $p \otimes p' \in \CA \otimes \CA' = \CB(\CH_\varphi)$ we get for all $n$
\begin{eqnarray*}
\theta \circ (Z')^n (p \otimes p') 
&=& \theta \circ (Z')^n (p \otimes \eins) - \theta \circ (Z')^n (p \otimes (\eins-p')) \\
&\ge& \theta \circ (Z')^n (p \otimes \eins) - \theta \circ (Z')^n (\eins \otimes (\eins-p')) \\
&\ge& \theta \circ (T^+)^n (p) - \theta \circ (T')^n  (\eins-p') 
> 1 - \frac{\epsilon}{2} - \frac{\epsilon}{2} = 1 - \epsilon. 
\end{eqnarray*}
Hence $\theta$ is tight for $Z'$. 

In (4)
the stationarity of $\varphi \otimes \psi$ 
for the automorphism $\alpha$ clearly implies $J: (\CB(\CH),\varphi) \rightarrow (\CB(\CH) \otimes \CC, \varphi \otimes \psi)$. It also implies that $\alpha$ commutes with the modular automorphism group of $\varphi \otimes \psi$, hence, because $\CA$ is globally invariant under this modular automorphism group, the same applies to $J(\CA)$. By a theorem of Takesaki, cf. \cite{Tak72}, this implies the existence of a $(\varphi \otimes \psi)$-preserving conditional expectation from $\CB(\CH) \otimes \CC$ onto $J(\CB(\CH))$. 
(In the terminology of \cite{Ku85} this follows from the fact that under the assumptions in (4) the automorphism $\alpha$ provides a dilation for the transition operator $T_\psi$, cf. \cite{Ku85}, 2.1.3.) Thus we have shown that the assumptions in (4) imply the assumptions in (3) and we have already proved above that in this case all normal states are tight for $Z'$.
\end{proof}

Finally we review the concept of asymptotic completeness for a transition. The discussion of its relationship with preparability will be the topic of the following section. 


\begin{Definition} [\cite{KM00}, \cite{GKL06}] \label{def:ac}
A transition $J: (\CA,\varphi) \rightarrow (\CA \otimes \CC, \varphi \otimes \psi)$,
with faithful normal states $\varphi$ and $\psi$, is called asymptotically complete if for all $a \in \CA$
\[
\lim_{n\to\infty} \| Q_n J_n(a) - J_n(a) \|_{\varphi \otimes \psi_n} = 0 \,,
\]
where $Q_n$ is the conditional expectation from $\CA \otimes \CC_n$ onto $\eins \otimes \CC_n \simeq \CC_n$
determined by $Q_n (a \otimes c) = \varphi(a) \, c$.
\end{Definition}

Asymptotic completeness was introduced for transitions induced by a coupling in \cite{KM00}. For its relation to scattering we refer to \cite{KM00}, further motivation and background can be found in \cite{WBKM00}, \cite{Go04}, and \cite{GKL06}. 
We shall see later, compare Corollary \ref{cor:ac-independent}, that surprisingly it turns out that asymptotic completeness in many cases does not depend on the choice of these faithful normal states. But for the moment we haven't proved this and we continue to work with Definition \ref{def:ac} and first state a few already well known properties of asymptotic completeness.


\begin{Lemma} \cite{GKL06}\label{lem:ac}
A transition $J: (\CA,\varphi) \rightarrow (\CA \otimes \CC, \varphi \otimes \psi)$,
with faithful normal states $\varphi$ and $\psi$,
is asymptotically complete
if and only if for all $a \in \CA$
\[
\lim_{n\to\infty} \| Q_n J_n(a) \|_{\psi_n} = \|a \|_\varphi \,.
\]
Asymptotic completeness implies that
\[
\lim_{n\to\infty} (T_{\psi})^n (a) = \varphi(a)\,\eins \;,
\]
in particular $J$ and $T_\psi$ are irreducible.
\end{Lemma}

The main reason for introducing the dual extended transition operator $Z'$ has been the following result.


\begin{Theorem} \cite{Go04,GKL06} \label{th:ac-Z}
For all $a \in \CA$ and all $n \in \Nset$
\[
\| Q_n J_n(a) \|^2_{\psi_n} \;=\; \langle a \eins, (Z')^n (p_\eins)\, a \eins \rangle_\varphi \,,
\]
where $p_\eins$ is the (one-dimensional) support of the vector state $\omega_\eins$.

A transition $J: (\CA,\varphi) \rightarrow (\CA \otimes \CC, \varphi \otimes \psi)$, with faithful normal states $\varphi$ and $\psi$, is asymptotically complete if and only if the vector
state $\omega_\eins$ is absorbing for $Z'$.
\end{Theorem}

In fact, the first formula explains how the dual extended transition operator $Z'$ can be used to prove the second statement.
See \cite{GKL06} for a concise proof of this fundamental result and \cite{Go04} for a broader discussion of it.


\begin{Corollary} \label{cor:ac-tight}
If $J: (\CA,\varphi) \rightarrow (\CA \otimes \CC, \varphi \otimes \psi)$ is asymptotically complete then it is tight.
\end{Corollary}

\begin{proof}
Asymptotic completeness implies that $Z'$ has an absorbing state by Theorem \ref{th:ac-Z}. But then all normal states must be tight for $Z'$, by Theorem \ref{th:prok}.
\end{proof}


\section{Asymptotic Completeness and Universal Preparability}
\label{section:ac-up}

We come to our main result. Recall that the assumption of tightness is a very weak one which is automatic in many cases of interest (see Proposition \ref{prop:J-tight}). In particular the existence of the faithful normal states $\varphi$ and $\psi$ is often automatic and otherwise rather a technical assumption, there is no need for an experimentalist to prepare these specific states. In fact, one of the equivalent properties below is about universal preparation of normal states which can be chosen beforehand in any way. The theorem lists this property and a number of other properties which are
equivalent to asymptotic completeness and which are interesting both from a mathematical and a physical point of view. These properties are subsequently discussed further in this section.


\begin{Theorem} \label{th:main}
Let $J: (\CB(\CH),\varphi) \rightarrow (\CB(\CH) \otimes \CC,\varphi \otimes \psi)$ be a tight transition with faithful normal states $\varphi$ on $\CB(\CH)$ and $\psi$ on $\CC$ (as in Definition \ref{def:J-tight}). Then the following assertions are equivalent:
\begin{itemize}
\item[(a)]
$J: (\CB(\CH),\varphi) \rightarrow (\CB(\CH) \otimes \CC,\varphi \otimes \psi)$
is asymptotically complete (cf. Definition \ref{def:ac}).
\item[(b1)]
All normal states on $\CB(\CH)$ are universally $J$-preparable (cf. Definition \ref{def:preparation}). 
\item[(b2)]
The universally $J$-preparable normal states are separating, in the sense that if $a \in \CB(\CH)$ satisfies $\rho(a)=0$ for all universally $J$-preparable normal states $\rho$ on $\CB(\CH)$ then $a=0$.
\item[(c)]
$J$ is topologically transitive (cf. Definition \ref{def:preparation}).
\item[(d1)]
The transition $J$ is irreducible and the map
\[
\CB(\CH) \ni a \mapsto \big( Q_n J_n(a) \big)_{n \in \Nset}\,
\] 
is injective.
\item[(d2)]
The map
\[
\CB(\CH) \ni a \mapsto \big( Q_n J_n(a) \big)_{n \in \Nset}\,
\] 
is isometric 

\end{itemize}
Here $Q_n$ is the conditional expectation from $\CB(\CH) \otimes \CC_n$ onto $\eins \otimes \CC_n \simeq \CC_n$
determined by $Q_n (a \otimes c) = \varphi(a) \, c$.
\end{Theorem} 

\begin{proof}
(a) $\Rightarrow$ (b1).
By Theorem \ref{th:up} it is enough to show that all vector states are
$J$-preparable by compatible preparing sequences from the faithful normal state $\varphi$ fixed in Theorem \ref{th:main}.  
Let $\xi$ be a unit vector in $\CH$, so $\omega_\xi$ is a vector state with (one-dimensional) support
projection $p_\xi$.
We start by defining positive linear functionals $\tilde{\theta}_n$ on
$\CC_n = \bigotimes^n_{j=1} \CC_{(j)}$ by
\[
\tilde{\theta}_n (c_n) := \psi_n \big( Q_n J_n(p_\xi) \cdot c_n \cdot Q_n J_n(p_\xi) \big) \quad\quad (c_n \in \CC_n).
\]
Then by normalization,
i.\,e., $\theta_n := \frac{\tilde{\theta}_n}{\|\tilde{\theta}_n\|}$
(for all $n \in \Nset$), we obtain a sequence $(\theta_n)$, where
$\theta_n$ is a normal state on $\CC_n$. We claim that $(\theta_n)$ is a preparing sequence from $\varphi$ to $\omega_\xi$.

Starting with the unnormalized case we consider for $a \in \CB(\CH)$ the sequence
\begin{eqnarray*}
(\varphi \otimes \tilde{\theta}_n)  J_n(a) 
&=& (\varphi \otimes \psi_n) \big(Q_n J_n(p_\xi) \cdot J_n(a) \cdot Q_n J_n(p_\xi) \big) \\
&=& (\varphi \otimes \psi_n) \big([Q_n J_n(p_\xi) - J_n(p_\xi)] \cdot J_n(a) \cdot Q_n J_n(p_\xi) \big) \\
&&+ \;(\varphi \otimes \psi_n) \big(J_n(p_\xi) \cdot J_n(a) \cdot [Q_n J_n(p_\xi) - J_n(p_\xi)] \big) \\
&&+ \;(\varphi \otimes \psi_n) \big( J_n(p_\xi) \cdot J_n(a) \cdot J_n(p_\xi) \big).
\end{eqnarray*}
Asymptotic completeness gives us that
\[
\lim_{n\to\infty} \| Q_n J_n(p_\xi) - J_n(p_\xi) \|_{\varphi \otimes \psi_n} = 0 \,,
\]

hence by using the Cauchy-Schwarz inequality (for the norms
$\|\cdot\|_{\varphi \otimes \psi_n}$) we see that the first and the second summand tend to $0$ for $n\to\infty$. For the third summand we find 
\[
(\varphi \otimes \psi_n) \big( J_n(p_\xi) \cdot J_n(a) \cdot J_n(p_\xi) \big)
= (\varphi \otimes \psi_n) \big( J_n(p_\xi a p_\xi) \big)
= \varphi (p_\xi a p_\xi) = \omega_\xi(a) \,\varphi(p_\xi),
\]
the last step follows from $p_\xi a p_\xi = \omega_\xi(a)\, p_\xi$
for the one-dimensional projection $p_\xi$. Hence
\[
\lim_{n\to\infty} (\varphi \otimes \tilde{\theta}_n)  J_n(a)
= \omega_\xi(a) \,\varphi(p_\xi).
\]
In particular for $a=\eins$ 

\[
\|\tilde{\theta}_n\| = \tilde{\theta}_n(\eins) 
= (\varphi \otimes \tilde{\theta}_n) \eins \otimes \eins
= (\varphi \otimes \tilde{\theta}_n) J_n(\eins) 
\to \omega_\xi( \eins)\, \varphi(p_\xi) = \varphi(p_\xi)  \quad \text{for}
\; n \to \infty.
\]
Together we obtain for $\theta_n = \frac{\tilde{\theta}_n}{\|\tilde{\theta}_n\|}$ (and for all $a \in \CB(\CH)$)
\[
\lim_{n \to \infty} (\varphi \otimes \theta_n)  J_n(a) = \omega_\xi(a),
\]
which shows that $\omega_\xi$ is $J$-preparable from $\varphi$ by the preparing sequence $(\theta_n)$, as claimed. These preparing sequences $(\theta_n)$ for different $\omega_\xi$ are compatible.

(b1) $\Rightarrow$ (b2) and (b1) $\Rightarrow$ (c) are obvious.

(b2) $\Rightarrow$ (d1) and (c) $\Rightarrow$ (d1).

We first show that both assumptions imply irreducibility of $J$.
Suppose $J$ is not irreducible. Then there exists a non-trivial projection $p \in \CB(\CH)$ with $J(p) \le p \otimes \eins$ (with $\eins \in \CC$). Then also $J_n(p) \le p \otimes \eins$ (with $\eins \in \CC_n$, use the iterative definition of $J_n$ and note that $J$ as a $*$-homomorphism is completely positive) and for every normal state $\sigma$ with support orthogonal to $p$ and every sequence $(\theta_k)$, where $\theta_k$ is a normal state on $\CC_{n_k}$, we have
\[
0 \le (\sigma \otimes \theta_k) J_{n_k}(p) \le (\sigma \otimes \theta_k) (p \otimes \eins) = \sigma(p) = 0\,. 
\]
Hence $(\sigma \otimes \theta_k) J_{n_k}(p) = 0$ for all $k$ and it follows that all states which are $J$-preparable from $\sigma$ must have support orthogonal to $p$.
This contradicts both (b2) and (c). We conclude that both (b2) and (c) imply that $J$ is irreducible.

Now let $0 \not= a \in \CB(\CH)$. If we have (b2) then there is a universally $J$-preparable normal state $\rho$ on  $\CB(\CH)$
such that $\rho(a) \not= 0$. If we have (c) then choose any normal state $\rho$ such that $\rho(a) \not= 0$. Both (b2) and (c) imply that there exists a normal state $\rho$ with $\rho(a) \not= 0$ which is $J$-preparable from $\varphi$ by a preparing sequence $(\theta_k)$. 
Then for $k$ big enough $| (\varphi \otimes \theta_k) J_{n_k} (a) - \rho(a) | < | \rho(a) |$,
hence $0 \not= (\varphi \otimes \theta_k) J_{n_k} (a) = \theta_k (Q_{n_k} J_{n_k} (a))$ and $Q_{n_k} J_{n_k} (a) \not= 0$. This proves that the map
$\CB(\CH) \ni a \mapsto \big( Q_n J_n(a) \big)_{n \in \Nset}$ 
is injective.

(d1) $\Rightarrow$ (a).
By Theorem \ref{th:ac-Z} it is enough to prove that the stationary vector state $\omega_\eins$ is absorbing for the dual extended transition operator $Z'$. 
Because by assumption $J$ is tight and hence all normal states on $\CB(\CH_\varphi)$ are tight for $Z'$ it is enough to prove, by Proposition \ref{prop:absorbing-criteria}(i), that there exists no other stationary normal state
except $\omega_\eins$. Suppose there is another stationary normal state. Then its support is not contained in the one-dimensional support $p_\eins$ of $\omega_\eins$. By 
Proposition \ref{prop:orthogonal} there exists a stationary normal state $\eta$ with support orthogonal to $p_\eins$. Because $J$ is irreducible 
we can represent $\eta$ as in Lemma \ref{lem:Z-stationary} and we find $a_j \in \CB(\CH)$ so that
\[
0 = \eta(p_\eins) = \eta \circ (Z')^n(p_\eins) 
= \sum^\infty_{j=1} \langle a_j \eins, (Z')^n(p_\eins) \,a_j \eins \rangle_\varphi,
\]
hence $\langle a_j, (Z')^n(p_\eins) \,a_j \rangle_\varphi = 0$, for all $n$ and all $j$. 
From Theorem \ref{th:ac-Z} we also get
\[
\langle a_j \eins, (Z')^n(p_\eins) \,a_j \eins \rangle_\varphi = \| Q_n J_n(a_j) \eins \|^2_{\psi_n}\,.
\]
so we have $Q_n J_n(a_j) = 0$ for all $n$ and all $j$. By injectivity in (d1) we conclude that $a_j=0$ for all $j$.
But this implies $\eta=0$ which is a contradiction.

The equivalence of (a) and (d2) is already stated in Lemma
\ref{lem:ac}. Property (d2) is included here to make explicit that in the situation of Theorem \ref{th:main} this strengthening of (d1) is actually equivalent to it.
\end{proof}

\begin{Remark}\normalfont
For finite dimensional $\CH$ the implication 
(d1) $\Rightarrow$ (a) was shown in \cite{KM00}, Theorem 4.2. For infinite dimensional $\CH$, however, considerably more effort is necessary.
\end{Remark}


The proof of (a) $\Rightarrow$ (b1), together with Proposition \ref{prop:faithful}, shows that we have actually established the existence of preparing sequences of a specific form.

\begin{Corollary} \label{rem:old-def}
If the conditions of Theorem \ref{th:main} are satisfied 
then all normal states on $\CB(\CH)$ are universally $J$-preparable with compatible universally preparing sequences $(\theta_n)$, so that $\theta_n$ is always a state on $\CC_n$ for all $n \in \Nset$. In particular, if $\sigma, \rho$ are normal states on $\CB(\CH)$ then there always exists a preparing sequence $(\theta_n)$ from $\sigma$ to $\rho$ where $\theta_n$ is a normal state on $\CC_n$ for all $n$.
\end{Corollary}

While the definition of preparability in \cite{Ha06} was based on
this observation we opted in Definition \ref{def:preparation} for a more flexible approach.

\begin{Remark}\normalfont
Because, by Corollary \ref{cor:ac-tight}, asymptotic completeness implies tightness of $J: (\CB(\CH),\varphi) \rightarrow (\CB(\CH) \otimes \CC,\varphi \otimes \psi)$
we could have stated Theorem \ref{th:main} also in the form that  
asymptotic completeness is equivalent to tightness plus one (and hence all) of the properties (b1), (b2), (c), (d1), (d2). We did state
Theorem \ref{th:main} in the way we did because in applications tightness is a property which is often automatically satisfied, see also Proposition \ref{prop:J-tight}. 
\end{Remark}

We already discussed asymptotic completeness, universal preparability, and topological transitivity earlier in this paper. Let us add a few comments on condition (d1) in Theorem \ref{th:main}. We can make its meaning more explicit by using some terminology from system theory and control theory, compare for example \cite{Ba75,Co07}. Roughly, a system is called controllable if it is always possible to drive any initial state of a system to a prescribed target state by using a suitable sequence of inputs. In this sense universal preparability of all normal states as in Theorem \ref{th:main}(b1) is a version of controllability. 

Dual to the concept of controllability is the concept of observability.
Roughly, a system is called observable if it is possible to determine the initial state of a system from observations of a suitable sequence of outputs. This may be possible even if the system itself cannot be observed directly. Observability can often be expressed by the injectivity of an observability map. The map introduced in Theorem \ref{th:main}(d1) looks formally like an observability map but it has observables as arguments instead of states. The relation between controllability and observability will be clarified in Theorem \ref{thm:observable}.


\begin{Definition} \label{def:observable}
Let the transition $J: \CA \rightarrow \CA \otimes \CC$ be given as $J(x) = \alpha(x \otimes \eins)$ with a coupling automorphism $\alpha$ of $\CA \otimes \CC$.
If for the sequence $(\theta_k)$, where $\theta_k$ is a normal state on $\CC_{n_k}$, the observability map
\[
\CA_*  \ni \sigma \mapsto \big(\, (\alpha_{n_k})_*(\sigma \otimes \theta_k)\, |_{\eins \otimes \CC_{n_k}} \,\big)_{k \in \Nset}
\]
is injective 
then we say that $\CA_*$ is observable by $\alpha$ and $(\theta_k)$.
\end{Definition}

We have given this definition in such a way that a kind of duality with 
the characterization of preparability in Theorem \ref{th:main}(d1)
becomes visible.
Intuitively the observability of $\CA_*$ means that if for monitoring the state evolution we only have access to observables in the algebras $\eins \otimes \CC_{n_k} \simeq \CC_{n_k}$ for all $k$ we can nevertheless determine the initial state $\sigma \in \CA_*$ from that.

The following result relates asymptotic completeness with this notion of observability. But note that it is the reverse transition $J^r$ associated to the inverse $\alpha^{-1}$ of the coupling automorphism and not $J$ itself which appears here. This should not come as a surprise if we compare it with the classical result in linear system theory where a system is controllable if and only if the dual system is observable, see for example \cite{Ba75}, Theorem 4.11, or \cite{Co07}. There are further discussions of duality in our setting in \cite{Go04} and \cite{Wi07}, let us also mention analogues of Kalman type algebraic controllability and observability criteria for asymptotic completeness in the language of multivariate operator theory in \cite{Go09,Go15}.

Before stating the theorem let us start with a few preliminary considerations about transitions and preduals. 
Let $\varphi$ be a faithful state in the predual $\CA_*$. 
If $a \in \CA$ then we define the functional $\varphi_a \in \CA_*$ by $\varphi_a(x) := \varphi(ax)$. By the Cauchy-Schwarz inequality we have $\| \varphi_a \| \le \|a\|_\varphi$, where $\|\cdot\|$ is the usual norm on the predual $\CA_*$.
Hence we can think of $a \mapsto \varphi_a$ as a continuous embedding of $\CA$ into $\CA_*$. 
It is a fact that $\CA$ is norm-dense in $\CA_*$, see \cite{Tak79}, III.2.7(iii).

Let a transition $J: \CA \rightarrow \CA \otimes \CC$ be given by $J(x) = \alpha(x \otimes \eins)$ with a coupling automorphism $\alpha$ of $\CA \otimes \CC$ such that with faithful normal states
$\varphi$ on $\CA$ and $\psi$ on $\CC$ the product state
$\varphi \otimes \psi$ is stationary for $\alpha$.

\begin{Lemma} \label{lem:J-psi}
 The map
\[
J^\psi \colon \CA_* \rightarrow (\CA \otimes \CC)_*, \quad \sigma \mapsto (\sigma \otimes \psi) \alpha^{-1}
\]
satisfies the following properties:
\begin{itemize}
\item[(1)]
$J^\psi$ is an extension of $J$ in the sense that for all $a \in \CA$
\[
J^\psi \varphi_a = (\varphi \otimes \psi)_{J(a)}.
\]
\item[(2)]
For all $\sigma \in \CA_*$ and $x \in \CA$
\[
(J^\psi \sigma)(J(x)) = \sigma(x).
\]
\item[(3)]
$J^\psi$ is an isometry , i.\,e., for all $\sigma \in \CA_*$
\[
\| J^\psi \sigma \| = \| \sigma \|\,.
\]
\end{itemize}
\end{Lemma}

\begin{proof}
Property (1) is shown by the following computation: For all $z \in \CA \otimes \CC$
\[
J^\psi \varphi_a (z) = (\varphi_a \otimes \psi) (\alpha^{-1}(z)) = (\varphi \otimes \psi) (a\! \otimes\! \eins \: \alpha^{-1}(z))
\]
\[
= (\varphi \otimes \psi) (\alpha(a \otimes \eins)\, z) = (\varphi \otimes \psi) (J(a)\, z) = (\varphi \otimes \psi)_{J(a)}(z).
\]
For (2) note that $J^\psi$ is continuous, in fact $ \| J^\psi \sigma \| \le \| \sigma \|$ is clear from the definition of $J^\psi$. Hence it is enough to check (2) on the dense subspace of functionals of the form $\varphi_a$ with $a \in \CA$. Indeed, using (1) we get
\[
J^\psi \varphi_a (J(x)) = (\varphi \otimes \psi)_{J(a)}(J(x)) = (\varphi \otimes \psi)(J(ax)) = \varphi(ax)  = \varphi_a(x).
\]
We already noted above that $J^\psi$ is contractive, the other direction needed for (3) follows from
\[
\|J^\psi \sigma\| = \sup_{z \in \CA \otimes \CC, \|z\|=1} | J^\psi \sigma (z) |
\ge \sup_{x \in \CA, \|x\|=1} | J^\psi \sigma (J(x)) |
= \sup_{x \in \CA, \|x\|=1} | \sigma(x) | = \| \sigma \|\,.
\]
Alternatively, (3) follows directly from the facts that $\varphi \otimes \psi$ is stationary for $\alpha^{-1}$ and that $\|\cdot\|$ is a cross-norm on tensor products of preduals (see  \cite{Tak79}, IV.5).
\end{proof}

Of course we can similarly extend the iterated transitions $J_n: \CA \rightarrow \CA \otimes \CC_n$  given by
$J_n(x) = \alpha_n (x \otimes \eins)$ and we obtain
$J^{\psi_n}_n \colon \CA_* \rightarrow (\CA \otimes \CC_n)_* $ 
where $\psi_n$ is the product state formed by copies of $\psi$ on $\CC_n$.
Further let $Q_n$ be the conditional expectation from $\CA \otimes \CC_n$ onto $\eins \otimes \CC_n \simeq \CC_n$
determined by $Q_n (a \otimes c) = \varphi(a) \, c$.


\begin{Lemma} \label{lem:observable}
 For all $a \in \CA$ and for all $\sigma \in \CA_*$ we have
\begin{eqnarray*}
\| Q_n J_n(a) \|_{\psi_n} &=& \sup_{c \in \CC_n, \|c\|_{\psi_n}=1}
|((\alpha^{-1})_{n})_* (\varphi_a \otimes \psi_n)\, (\eins \otimes c)| \,, \\
\| (J^{\psi_n}_n \sigma) |_{\eins \otimes \CC_n} \| 
&=&\;\; \sup_{c \in \CC_n, \|c\|=1}
|((\alpha^{-1})_{n})_* (\sigma \otimes \psi_n\,)\, (\eins \otimes c)|
\,.
\end{eqnarray*}
\end{Lemma}

\begin{proof}
For the following computation
recall the formula $(\alpha_n)^{-1} = (\id \otimes R_n) \circ
(\alpha^{-1})_n \circ (\id \otimes R_n)$ from the beginning of Section \ref{section:preparation} and note that $\varphi \otimes \psi_n$ is stationary for the automorphisms $\alpha_n$ and $\id \otimes R_n$.
Further $(R_n)^2 = \id$.
For all $a \in \CA$ and $c \in \CC_n$ we have
\[
 ((\alpha^{-1})_{n})_* (\varphi_a \otimes \psi_n) (\eins \otimes c) 
= (\varphi_a \otimes \psi_n) \big[(\alpha^{-1})_n (\eins \otimes c) \big] 
= (\varphi \otimes \psi_n) \big[ a \otimes \eins \,\cdot\, (\alpha^{-1})_n (\eins \otimes c) \big] 
\]
\[
= (\varphi \otimes \psi_n)\circ(\id \otimes R_n) \big[ a \otimes \eins \,\cdot\, (\alpha^{-1})_n (\id \otimes (R_n)^2) (\eins \otimes c) \big]
\]
\[
= (\varphi \otimes \psi_n) \big[ a \otimes \eins \,\cdot\, (\id \otimes R_n) (\alpha^{-1})_n (\id \otimes R_n) (\eins \otimes R_n(c)) \big] 
\]
\[
= (\varphi \otimes \psi_n) \big[ \alpha_n(a \otimes \eins) \,\cdot\,(\eins \otimes R_n(c)) \big] \\
= (\varphi \otimes \psi_n) \big[ J_n(a) \, \cdot\,(\eins \otimes R_n(c)) \big] 
= \psi_n \big[ Q_n J_n(a) \,\cdot\, R_n(c) \big].
\]
Now (as $\psi_n$ is  stationary for the automorphism $R_n$) we conclude that
\[
 \sup_{c \in \CC_n, \|c\|_{\psi_n}=1}
|((\alpha^{-1})_{n})_* (\varphi_a \otimes \psi_n) (\eins \otimes c)| 
= \| Q_n J_n(a) \|_{\psi_n}\,,
\]
which is the first formula in Lemma \ref{lem:observable}.

\noindent
For the second formula we can proceed similarly:
\[
((\alpha^{-1})_{n})_* (\sigma \otimes \psi_n\,)\, (\eins \otimes c)
= (\sigma \otimes \psi_n\,) \big[  (\id \otimes R_n) (\alpha_n)^{-1} (\id \otimes R_n) (\eins \otimes c) \big]
\]
\[
= (\sigma \otimes \psi_n\,) \big[  (\alpha_n)^{-1}  (\eins \otimes R_n(c)) \big]
=  J^{\psi_n}_n \sigma (\eins \otimes R_n(c))
\]
and taking the $\sup$ yields the result.
\end{proof}

We are now ready to state and prove the announced result on observability.

\begin{Theorem} \label{thm:observable}
Let the transition $J: \CA \rightarrow \CA \otimes \CC$ be given by $J(x) = \alpha(x \otimes \eins)$ with a coupling automorphism $\alpha$ of $\CA \otimes \CC$ such that with faithful normal states
$\varphi$ on $\CA$ and $\psi$ on $\CC$ the product state
$\varphi \otimes \psi$ is stationary for $\alpha$.
\begin{itemize}
\item[(1)]
$J: (\CA,\varphi) \rightarrow (\CA \otimes \CC, \varphi \otimes \psi)$ is asymptotically complete if and only if
\[
 \lim_{n \to \infty} \;\sup_{c \in \CC_n, \|c\|_{\psi_n}=1}
|((\alpha^{-1})_{n})_* (\varphi_a \otimes \psi_n)\, (\eins \otimes c)| 
= \| a \|_{\varphi} \quad \text{for all} \;\; a \in \CA\,.
\]
\item[(2)]
If $J: (\CA,\varphi) \rightarrow (\CA \otimes \CC, \varphi \otimes \psi)$ is asymptotically complete then
\[
\lim_{n \to \infty} \sup_{c \in \CC_n, \|c\|=1}
|((\alpha^{-1})_{n})_* (\sigma \otimes \psi_n)\, (\eins \otimes c)| 
= \| \sigma \| \quad \text{for all} \;\; \sigma \in \CA_*
\]
and $\CA_*$
is observable by $\alpha^{-1}$ and the product states $(\psi_n)$.
\item[(3)]
Suppose that $\CA = \CB(\CH)$. Then the following are equivalent:
\begin{itemize}
\item[(a)]
$J: (\CB(\CH),\varphi) \rightarrow (\CB(\CH) \otimes \CC, \varphi \otimes \psi)$ is asymptotically complete.
\item[(b)]
$J$ is irreducible and  
$\CB(\CH)_*$ is observable by $\alpha^{-1}$ and $(\psi_n)$.
\end{itemize}
\end{itemize}
\end{Theorem} 

The limit formulas given in Theorem \ref{thm:observable} should be interpreted as uniformity properties of the observability map,
in fact they show that this map is isometric for the norms indicated. 
Clearly such uniformity is essential to make practical use of observability in the design of physical experiments: we need information about the precision needed in measuring observables in $\CC_n$ in order to distinguish between states on $\CA$. 
So it is remarkable that such uniformity is automatic in the case of asymptotic completeness.

\begin{proof}
From Lemma \ref{lem:observable} we have
\[
\| Q_n J_n(a) \|_{\psi_n} = \sup_{c \in \CC_n, \|c\|_{\psi_n}=1}
|((\alpha^{-1})_{n})_* (\varphi_a \otimes \psi_n)\, (\eins \otimes c)|
\]
and by Lemma \ref{lem:ac} this converges to $\|a\|_\varphi$ for all $a \in \CA$ (for $n \to \infty$) if and only if 
$J: (\CA,\varphi) \rightarrow (\CA \otimes \CC, \varphi \otimes \psi)$ is asymptotically complete.
This gives (1). For (2) note that in view of
\[
\| (J^{\psi_n}_n \sigma) |_{\eins \otimes \CC_n} \| 
=\; \sup_{c \in \CC_n, \|c\|=1}
|((\alpha^{-1})_{n})_* (\sigma \otimes \psi_n\,)\, (\eins \otimes c)|
\]
from Lemma \ref{lem:observable} we have to prove that asymptotic completeness implies that for all $\sigma \in \CA_*$
\[
\lim_{n \to \infty} \| (J^{\psi_n}_n \sigma) |_{\eins \otimes \CC_n} \| = \| \sigma \|\,.
\]
In fact, because $J^{\psi_n}_n$, by Lemma \ref{lem:J-psi}(3), is an isometry on preduals it is enough to prove this claim on the dense subspace of functionals $\sigma = \varphi_a$ with
$a \in \CA$. Hence it remains to show that for all $a \in \CA$
\[
\lim_{n \to \infty} \| (J^{\psi_n}_n \varphi_a) |_{\eins \otimes \CC_n} \| = \| \varphi_a \|\,.
\]
 First note that for $z \in \CA \otimes \CC_n$ we have (with Lemma \ref{lem:J-psi}(1) and using the fact that $Q_n$ is a conditional expectation)
\[
J^{\psi_n}_n \varphi_a (Q_n(z)) 
= \varphi \otimes \psi_n (J_n(a) \cdot Q_n(z))
= \varphi \otimes \psi_n (Q_n J_n(a) \cdot z)\,.
\]
Because $Q_n$ maps the unit ball of $\CA \otimes \CC_n$ onto the unit ball of $\eins \otimes \CC_n$ we conclude that
\[
\| (J^{\psi_n}_n \varphi_a) |_{\eins \otimes \CC_n} \| 
= \| (\varphi \otimes \psi_n)_{Q_n J_n(a)}\|\,.
\]
By Cauchy-Schwarz we have $\| \varphi_a \| \le \| a \|_\varphi$ and instead for $a$ we apply this for $Q_n J_n(a) - J_n(a)$. Now we can use asymptotic completeness in the form of Lemma \ref{lem:ac} and obtain
\[
\| (\varphi \otimes \psi_n)_{Q_n J_n(a)} 
- (\varphi \otimes \psi_n)_{J_n(a)}\|
\le \| Q_n J_n(a) - J_n(a) \|_{\varphi\otimes \psi_n}
\to 0 \quad (n \to \infty)\,.
\]
But with Lemma \ref{lem:J-psi}(1) and (3) we get (for all $n$)
\[
\| (\varphi \otimes \psi_n)_{J_n(a)} \| 
= \| J^{\psi_n}_n \varphi_a \| = \| \varphi_a \|\,,
\]
hence also
\[
\| (J^{\psi_n}_n \varphi_a) |_{\eins \otimes \CC_n} \| = \| (\varphi \otimes \psi_n)_{Q_n J_n(a)} \|  \to \| \varphi_a \|
\quad (n \to \infty)\,,
\]
which proves our claim.

It is clear that this implies the observability of $\CA_*$ by $\alpha^{-1}$ and the product states $(\psi_n)$, in fact the limit formula valid for all $\sigma \in \CA_*$ beyond injectivity even provides an additional uniformity property of the observability map.

In (3) the implication $(a) \Rightarrow (b)$ follows from (2) and Lemma \ref{lem:ac}. For the implication $(b) \Rightarrow (a)$
suppose now that $\CA = \CB(\CH)$. 
Note that $J$ is tight,
by Proposition \ref{prop:J-tight}(4), so we are in the setting of Theorem \ref{th:main}.
Consider again the equality
\[
\| Q_n J_n(a) \|_{\psi_n} = \sup_{c \in \CC_n, \|c\|_{\psi_n}=1}
|((\alpha^{-1})_{n})_* (\varphi_a \otimes \psi_n)\, (\eins \otimes c)| 
 \quad \text{for all} \;\; a \in \CA\,,
\]
from Lemma \ref{lem:observable}. 
If $\CB(\CH)_*$ is observable by $\alpha^{-1}$ and the product states $(\psi_n)$ 
then it follows that for all $0 \not=a \in \CA$ this expression must be strictly positive for some $n$. If $J$ is (by assumption) also irreducible then we 
have verified (d1) of  Theorem \ref{th:main} which now implies that $J: (\CB(\CH),\varphi) \rightarrow (\CB(\CH) \otimes \CC, \varphi \otimes \psi)$ is asymptotically complete. 
\end{proof}


Theorem \ref{th:main} contains further interesting insights about asymptotic completeness if we combine it with earlier results about
tightness in Proposition \ref{prop:J-tight}. We work out the most important case, namely transitions on $\CB(\CH)$ induced by coupling automorphisms.

\begin{Corollary} \label{cor:ac-independent}
Let the transition $J: \CB(\CH) \rightarrow \CB(\CH) \otimes \CC$ be induced by a coupling automorphism $\alpha$ of $\CB(\CH) \otimes \CC$ such that $J(a) = \alpha(a \otimes \eins)$ for all $a \in \CB(\CH)$. Suppose further that $\varphi_1, \varphi_2$ on $\CB(\CH)$ and $\psi_1, \psi_2$ on $\CC$ are faithful normal states such that $\varphi_1 \otimes \psi_1$ and $\varphi_2 \otimes \psi_2$ are stationary for the automorphism $\alpha$. Then $J: (\CB(\CH),\varphi_1) \rightarrow (\CB(\CH) \otimes \CC, \varphi_1 \otimes \psi_1)$ is asymptotically complete if and only if $J: (\CB(\CH),\varphi_2) \rightarrow (\CB(\CH) \otimes \CC, \varphi_2 \otimes \psi_2)$ is asymptotically complete.
\end{Corollary}

\begin{proof}
Both
$J: (\CB(\CH),\varphi_1) \rightarrow (\CB(\CH) \otimes \CC, \varphi_1 \otimes \psi_1)$ and $J: (\CB(\CH),\varphi_2) \rightarrow (\CB(\CH) \otimes \CC, \varphi_2 \otimes \psi_2)$ 
are tight by Proposition \ref{prop:J-tight}(4), so we are in the setting
of Theorem \ref{th:main}. But some of the assertions equivalent to asymptotic completeness in Theorem \ref{th:main}, for example
\ref{th:main}(c), do not involve the states $\varphi$ and $\psi$. 
\end{proof} 

This suggests a natural definition of asymptotic completeness for $\alpha$ itself.

\begin{Definition} \label{def:coup-ac}
We call a coupling automorphism $\alpha: \CB(\CH) \otimes \CC \rightarrow \CB(\CH) \otimes \CC$ asymptotically 
complete if there exist faithful normal states $\varphi$ on $\CB(\CH)$ and $\psi$ on $\CC$ such that $\varphi \otimes \psi$ is stationary for
$\alpha$ and for any choice of such states (and hence for all choices of such states) the induced transition 
$J: (\CB(\CH),\varphi) \rightarrow (\CB(\CH) \otimes \CC, \varphi \otimes \psi)$ is asymptotically complete.
\end{Definition}

Based on the other parts of Proposition \ref{prop:J-tight} we can in a similar way find other situations where asymptotic completeness is preserved under changes in the choice of the faithful normal states $\varphi$ and $\psi$. For example this follows for transitions of the form $J: \CB(\CH) \rightarrow \CB(\CH) \otimes \CB(\CK)$, with $\CH$ and $\CK$ finite dimensional, from Proposition \ref{prop:J-tight}(1).
In the finite dimensional case this remarkable phenomenon has already been observed in \cite{GKL06}, Proposition 4.4, where moreover it is shown that in this case the dual extended transition
operators corresponding to different choices of states are similar to each other. 

We finish this section with an easy example illustrating some further subtleties in the relationship between asymptotic completeness and preparability of states. 

\begin{Example} \normalfont
The simplest examples of asymptotically complete transitions are tensor flips, for example $J: M_2 \rightarrow M_2 \otimes M_2,\; a \mapsto \eins \otimes a$. In this case obviously only one step is needed to prepare any desired state. Let us modify this example and consider the transition  $J: M_2 \rightarrow M_2 \otimes M_2$ determined
by $J(b) := \sigma_x \otimes b$, where
$b := 
\left( \begin{array}{cc}
0 &  0 \\
1 &  0 \\
       \end{array} 
\right)$
and
$\sigma_x := 
\left( \begin{array}{cc}
0 &  1 \\
1 &  0 \\
       \end{array} 
\right)$.
This means that $J(a) = \eins \otimes a$ if $a$ is diagonal and $J(x) = \sigma_x \otimes a$ if $a$ is off-diagonal (i.\,e., with zero entries on the diagonal).
We have $J_n(\sigma_x) = \sigma_x \otimes \bigotimes^n_{j=1} \sigma_x$.
Choosing the tracial state we find $Q_n J_n(\sigma_x) = 0$ for all $n$ and condition (d1) of Theorem \ref{th:main} is not satisfied, hence $J$ is not asymptotically complete. But note that the vector states $\omega_{\delta_0}$ and $\omega_{\delta_1}$ from the canonical basis 
$\delta_0 = 
\left( \begin{array}{c}
 0 \\
 1 \\
       \end{array} 
\right)$
and 
$\delta_1 = 
\left( \begin{array}{c}
 1 \\
 0 \\
       \end{array} 
\right)$ of $\Cset^2$
are both universally $J$-preparable. In fact, this can still be done in one step by choosing $\omega_{\delta_0}$ resp. $\omega_{\delta_1}$ as the preparing state $\theta$ on $M_2$: it is easily checked that for any state $\rho$ on $M_2$ and $a \in M_2$ we have
\[
(\rho \otimes \omega_{\delta_j})\, J(a) = \omega_{\delta_j}(a) \quad\quad (j=0,1).
\]
This shows that universal preparability of all vector states corresponding to a basis is not sufficient for asymptotic completeness.
\end{Example}


\section{A Class of Quantum Birth and Death Chains}
\label{section:example}

In this section we apply our theory to an interesting example, a class of quantum birth and death chains which appears in quantum optics experiments involving repeated interactions, for example between a micromaser and a stream of atoms. We follow \cite{BGKRSS} where the setting is described and interpreted in more detail and further results about its properties can be found.

Consider the Hilbert space $\CH:=\ell^2(\Nset_0)$ with canonical orthonormal basis $(\delta_n)_{n=0}^\infty$ and the Hilbert space $\Cset^2$ with canonical orthonormal basis 
$\epsilon_0 := \left( \begin{array}{c}
 0 \\
 1 \\
       \end{array} 
\right)$
and $\epsilon_1 := \left( \begin{array}{c}
 1 \\
 0 \\
       \end{array} 
\right)$,
so that with $\delta_{n,0} := \delta_n\otimes \epsilon_0$ and $\delta_{n,1} := \delta_n\otimes \epsilon_1$ the vectors $(\delta_{n,\epsilon})_{n \in \Nset_0, \epsilon \in \{0,1\}}$ form an orthonormal basis of $\CH \otimes \Cset^2$. For $n\in \Nset_0$ consider the {\it n-particle space} $\CH_n$ where $\CH_0$ is spanned by $\delta_0\otimes \epsilon_0$ and for $n \ge 1$ the subspace $\CH_n$ is spanned by $\{\delta_{n-1}\otimes \epsilon_1, \delta_n\otimes \epsilon_0\}$. Then 
$\CH\otimes \Cset^2 = \bigoplus_{n\in \Nset_0} \CH_n$. 


On $\CH \otimes \Cset^2$ we consider a unitary $u$ which leaves the subspaces $\CH_n$ ($n \in \Nset_0$) invariant. Then $u$ on $\CH_0$ is multiplication by a complex number $\alpha_0$ with $|\alpha_0|=1$ while on $\CH_n$ with $n \ge 1$ it is given by a unitary $2\times 2$ matrix which we conveniently denote by 
\[u_n =  \left( \begin{array}{cc}
               \alpha^+_n &  \beta^+_n \\
 \beta_n &  \alpha_n \\
       \end{array} 
\right)\,,
\]
such that 

\begin{eqnarray*}
u_n\, \delta_{n-1,1} &=& \alpha^+_n\, \delta_{n-1,1} + \beta_n\,\delta_{n,0}\, ,\\
u_n\, \delta_{n,0}   &=& \beta^+_n\,\delta_{n-1,1} +\alpha_n\, \delta_{n,0} \,.
\end{eqnarray*}

\noindent
Since $u_n$ is unitary we obtain the relations $|\alpha_n|=|\alpha_n^+|$, 
$|\beta_n|=|\beta_n^+|$, and $|\alpha_n|^2 + |\beta_n|^2 = 1$ for all $n \ge 1$.

\noindent
The unitary $u$ defines a coupling automorphism
\begin{eqnarray*}
\alpha:\CB(\CH) \otimes M_2 &\to& \CB(\CH)\otimes M_2,\\
   x\otimes y &\mapsto& u^*\,x\otimes y\, u
\end{eqnarray*}
which in turn induces a transition $J:\CB(\CH) \to \CB(\CH)\otimes M_2$ by $J(x):= \alpha(x\otimes \eins)$. We call $\alpha$ a {\it generalized micromaser coupling} and $J$ a {\it generalized micromaser transition} (cf. the discussion below).


On $M_2$ we consider, for $0 \le \lambda \le 1$, the state $\psi_\lambda$ given by 
$\psi_\lambda (\cdot) := \Trace(d_\lambda \cdot)$ with $d_\lambda :=
\left( \begin{array}{cc}
\lambda & 0 \\
0 & 1-\lambda \\
     \end{array} 
\right).
$
For $0 \le \lambda < \frac12$ we define the state $\varphi_\lambda$ on $\CB(\CH)$ which is given by the diagonal density matrix $\nu_\lambda$ with $\nu_\lambda(\delta_n)=\frac{1-2\lambda}{1-\lambda}(\frac{\lambda}{1-\lambda})^n\;\delta_n$ $(n\in \Nset_0)$. 

Since $\delta_{n-1,1}$ and $\delta_{n,0}$ are both eigenvectors of $\nu_\lambda \otimes d_\lambda$ with the same eigenvalue $\frac{1-2\lambda}{1-\lambda}\frac{\lambda^n}{(1-\lambda)^{n-1}}$,
the density matrix $\nu_\lambda \otimes d_\lambda$ is constant on the subspaces $\CH_n$ ($n\in \Nset$) and thus it commutes with the unitary $u$. Therefore, for $0 \le \lambda < \frac 12$ the state $\varphi_\lambda\otimes\psi_\lambda$ is invariant under the coupling automorphism $\alpha$ (for $\lambda \ge \frac 12$ it would not define a state at all). We summarize these considerations as follows.

\begin{Proposition}[cf. \cite{BGKRSS}, Proposition 4.4] \label{prop:lambda}
Let $J$ be a generalized micromaser transition. Then for all $0 \le \lambda < \frac{1}{2}$, with the states $\varphi_\lambda$ and $\psi_\lambda$ defined above, we have
\[
J: (\CB(\CH),\varphi_\lambda) \rightarrow (\CB(\CH),\varphi_\lambda) \otimes (M_2, \psi_\lambda).
\]
Clearly, the states $\varphi_0$ and $\psi_0$ are vector states, while 
for $0 < \lambda < \frac{1}{2}$ the states
$\varphi_\lambda$ and $\psi_\lambda$ are faithful.
\end{Proposition}


\begin{Remark} \normalfont  \label{rem:jaynes-cummings} These kinds of couplings occur in dynamics resulting from Jaynes-Cummings type interactions in quantum optics. Most prominently such an interaction is found in the micromaser system, where one quantum of energy is exchanged between an incoming two level atom and a mode of the electromagnetic field in a cavity. Some consequences of asymptotic completeness for this system have been studied in \cite{WBKM00} and further information on this system may be found in \cite{MS91,HR06}. The same mathematical model also applies to a single ion in a trap or to a neutral atom in a laser trap (cf. \cite{HR06}).

In order to relate our discussion to these physical models we start with briefly reviewing the relevant formulas from \cite{HR06}, 3.4: `Coupling a spin and a spring: the Jaynes-Cummings model', and refer for more explanations and details to this book.

We have an atomic Hamiltonian $H_a = \frac{\hbar\, \omega_{eg}}{2} \sigma_z$ and a cavity Hamiltonian $H_c = \hbar \omega_c N$. Here 
$\sigma_z = 
\left( \begin{array}{cc}
1 &  0 \\
0 &  -1 \\
       \end{array} 
\right)$,
$N$ is the number operator of the harmonic oscillator, and $\omega_c$ and $\omega_{eg}$ are frequencies associated to the cavity and to the atom.
If the frequencies $\omega_c$ and $\omega_{eg}$ are close then the rotating wave approximation simplifies the atom-cavity coupling Hamiltonian to 
$H_{ac} = -i \hbar \frac{\Omega_0}{2} \big[ a \sigma_+ - a^\dagger \sigma_- \big]$, where $a, a^\dagger$ and $\sigma_-, \sigma_+$ are annihilation and creation operators for the cavity and for the atom and $\Omega_0$ is the vacuum Rabi frequency.

Thus one arrives at a total Hamiltonian
\[
H = H_a + H_c + H_{ac} = 
\frac{\hbar\, \omega_{eg}}{2} \sigma_z + \hbar \omega_c N +  
(-i \hbar \frac{\Omega_0}{2}) \big[ a \sigma_+ - a^\dagger \sigma_- \big]\, .
\]
The quantity $\Delta_c = \omega_{eg} - \omega_c$ is called atom-cavity detuning, the special case $\Delta_c = 0$ is called resonant while the effects of $\Delta_c \not=0$ are referred to as results of detuning. 

Eigenstates for $H_a + H_c$ are $|g,n\rangle, |e,n\rangle$, where $|n\rangle$ is the $n$-photon state of the cavity 
($n \in \Nset_0$) and $|g\rangle, \, |e\rangle$ are the ground state and the excited state of the atom. Because $H = H_a + H_c + H_{ac}$ preserves the excitation number the dynamics connects only states inside the doublets formed by $|e,n\rangle$ and $|g,n+1\rangle$ (and leaves the ground state $|g,0\rangle$ unchanged). More explicitly, with the so called $n$-photon Rabi frequency $\Omega_n = \Omega_0 \sqrt{n+1}$ and the angle $\theta_n$ given by $\tan \theta_n = \frac{\Omega_n}{\Delta_c}$ 
(or $\theta_n = \frac{\pi}{2}$ in the resonant case) the eigenstates of the coupled atom-cavity system (also called dressed states) are
given by
\begin{eqnarray*}
|+,n\rangle &=& \cos \frac{\theta_n}{2}\, |e,n\rangle + i \sin \frac{\theta_n}{2}\, |g,n+1\rangle, \\
|-,n\rangle &=& \sin \frac{\theta_n}{2}\, |e,n\rangle - i \cos \frac{\theta_n}{2}\, |g,n+1\rangle.
\end{eqnarray*}
The corresponding energies are $E^{\pm}_n = \frac{n+1}{2} \hbar \omega_c \pm \frac{\hbar}{2} \sqrt{\Delta^2_c +
\Omega^2_n}$ and the time evolution of any state, in particular of initial states $|e,n\rangle$ and $|g,n+1\rangle$, can now be explicitly computed. 

Consider now the micromaser experiment (for a more detailed discussion of this system we refer to \cite{WBKM00}): In this experiment single atoms are sent one after the other through a cavity. Only during their passage through the cavity they interact with the field mode inside. Within a good approximation all atoms pass the cavity with the same velocity and there is only one atom inside the cavity at a time. Therefore, the effect of the interaction of each of the atoms with the field is described by $e^{-iHT/\hbar}$ for a fixed effective interaction time $T$ determined by the velocity of the atoms and the size of the cavity. During this time there is a continuous reversible exchange of energy between $|e,n\rangle$ and $|g,n+1\rangle$ with a period determined by the corresponding Rabi frequency. It may happen, however, that for certain values of the Rabi frequency and the interaction time T the total effect of the interaction during the atoms passage through the cavity results in an integer number of Rabi oscillations. In this case no energy has been exchanged at the moment when the atoms leaves the cavity and the states 
$|e,n\rangle$ and $|g,n+1\rangle$ remain unaltered. Such a situation is referred to as a {\it trapped state condition}.

The micromaser is easily related to our previous discussion: There the vectors $|n\rangle$ are denoted by $\delta_n$ ($n \in \Nset_0)$, $|g\rangle$ and $|e\rangle$ are called $\epsilon_0$ and $\epsilon_1$, correspondingly $|g,n\rangle$ and $|e,n\rangle$ become $\delta_{n,0}$ and $\delta_{n,1}$. The total Hamiltonian $H$ leaves the n-particle spaces $\CH_n$ spanned by 
$|g,n\rangle$ and $|e,n-1\rangle$ ($n\in \Nset$) invariant, as well as the one-dimensional subspace spanned by the vacuum 
$|g,0\rangle$, and so does the unitary $e^{-iHT/\hbar}$ which corresponds to the unitary $u$ defined above. Thus the parameters $|\alpha_n|^2$, $|\alpha^+_n|^2$, $|\beta_n|^2$, and $|\beta^+_n|^2$ determine the resulting transition probabilities. In particular, a trapped state condition occurs exactly if one of the parameters $\beta_n$, $n \ge 1$, vanishes (in this case $\beta^+_n =0$, too). In this case asymptotic completeness is, clearly, impossible because certain transitions between energy levels are forbidden. 

If there is no trapping state, however, then we prove asymptotic completeness for our model in the following theorem. The above discussion shows that this result can be applied, in particular, to the Jaynes-Cummings model, both in the resonant and in the detuned case, as long as the rotating wave approximation is appropriate. We finally remark that from the discussion in \cite{WBKM00} it is immediate that the iterated transitions $J_n$ describe (in the Heisenberg picture) the effect of $n$ atoms having passed the cavity. 

\end{Remark}


\noindent
For the following discussion it is convenient to have an alternative description of the unitary $u$. We identify $\CH \otimes \Cset^2$ with $\CH \oplus \CH$ and $u \in \CB(\CH \otimes \Cset^2) = \CB(\CH) \otimes M_2 = M_2(\CB(\CH))$ with a $2\times 2$ block matrix with entries from $\CB(\CH)$.
On the Hilbert space $\CH = \ell^2(\Nset_0)$ we define the diagonal operators $a$ with diagonal $(\alpha_n)^\infty_{n=0}$, $a^+$ with diagonal $(\alpha^+_{n+1})^\infty_{n=0}$ (note that here the index is shifted by 1), $b$ with
diagonal $(\beta_n)^\infty_{n=0}$, and $b^+$ with
diagonal $(\beta^+_n)^\infty_{n=0}$, and we may put $\beta_0=0=\beta^+_0$. Let $s$ denote the right shift on
$\CH = \ell^2(\Nset_0)$, i.e., $s\, \delta_n = \delta_{n+1}$ for all $n$. Then $u$ is given by the $2\times 2$-block matrix

\[
u = 
\left( \begin{array}{cc}
 a^+ &  s^* b^+ \\
 b s &  a \\
       \end{array} 
\right)\,.
\]

\noindent
The input state $\psi_\lambda$ induces on $\CB(\CH)$ the transition operator 
$T_{\psi_\lambda} = P_{\psi_\lambda} J: \CB(\CH) \rightarrow \CB(\CH)$ which is given by
\[
T_{\psi_\lambda}(x)=\lambda\big((a^+)^* xa^+ + s^*b^*\,x\,bs\big) + (1-\lambda)\big((b^+)^*s\,x\,s^*b^+ + a^*xa\big) \, .
\]

\noindent
The analysis of the transition $J$ is simplified by the fact that the transition operator $T_{\psi_\lambda}$ leaves invariant the diagonal algebra $\CD \,(\simeq \ell^\infty(\Nset_0))$ which is obtained as the weak$^*$-closure of the linear span of the one-dimensional projections $(p_{\delta_n})^\infty_{n=0}$. The restriction $T_{\CD, \psi_\lambda}$ 
of $T_{\psi_\lambda}$ to the diagonal $\CD$
is the transition matrix of a classical birth and death chain which is given by
\begin{multline*}
T_{\CD, \psi_\lambda}
=\\ \hspace*{5mm} 
\left( \begin{array}{cccccc}
(1-\lambda) + \lambda |\alpha_1|^2 &  \lambda |\beta_1|^2  & 0 & 0 & 0 &\ldots  \\
(1-\lambda) |\beta_1|^2 & (1-\lambda)|\alpha_1|^2 + \lambda |\alpha_2|^2 &  \lambda |\beta_2|^2 & 0 & 0 &\ldots  \\
0 &  (1-\lambda)|\beta_2|^2  & (1-\lambda)|\alpha_2|^2 + \lambda |\alpha_3|^2 &  \lambda |\beta_3|^2  & 0 &\ldots  \\
 \vdots  & \vdots   & \vdots & \vdots & \vdots &\ddots \\                                   
       \end{array} 
\right).
\end{multline*}
Here we used the fact that $|\alpha_n|=|\alpha^+_n|$ and $|\beta_n|=|\beta^+_n|$ for $n\ge 1$.

We can now check (cf. \cite{BGKRSS}, Proposition 4.4, it also follows from the considerations above)  that 
for $0 \le \lambda < \frac{1}{2}$ there is a stationary probability measure $\nu_\lambda$ for $T_{\CD, \psi_\lambda}$ given by $\nu_\lambda(n) =
\frac{1-2\lambda}{1-\lambda}(\frac{\lambda}{1-\lambda})^n\; (n \in \Nset_0)$, i.\,e. $\nu_\lambda$ is the density of $\varphi_\lambda$.
\\ {\phantom{x}} 
\\
We now prove asymptotic completeness. We remark that asymptotic completeness of a finite dimensional cut-off version of this transition has been proved in \cite{GKL06}, Section 6. For the infinite dimensional version considered here we need the tools prepared in the previous sections.


\begin{Theorem} \label{th:micromaser} If in the unitary $u$ we have $\beta_n \not=0$ for all $n\in \Nset$, i.\,e., in the absence of trapping states, the generalized micromaser coupling $\alpha$ is asymptotically complete
(in the sense of Definition \ref{def:coup-ac}). 
\end{Theorem}

\begin{proof}
We have already seen in Proposition \ref{prop:lambda} that for all $0 \le \lambda < \frac 12$ the states $\varphi_\lambda \otimes \psi_\lambda$ are stationary for the generalized micromaser coupling $\alpha$.

Because for $0 < \lambda < \frac{1}{2}$ the states $\varphi_\lambda$ and $\psi_\lambda$ are faithful, we conclude in this case, by Proposition \ref{prop:J-tight}(4), that the corresponding
transition $J: (\CB(\CH), \varphi_\lambda) \rightarrow (\CB(\CH) \otimes M_2,
\varphi_\lambda \otimes \psi_\lambda)$ is tight. Hence 
we are in the setting of Theorem \ref{th:main} and we only need to check one of the equivalent conditions there to show that 
$J: (\CB(\CH), \varphi_\lambda) \rightarrow (\CB(\CH) \otimes M_2,
\varphi_\lambda \otimes \psi_\lambda)$ is asymptotically complete. In the light of Corollary \ref{cor:ac-independent} this also proves the theorem. 

We verify universal preparability of all normal states, condition \ref{th:main}(b1), with the help of the criterion given in Theorem \ref{th:reverse}.

We make use of the vector states 
$\varphi_0 = \omega_\xi$ with $\xi = \delta_0$ on $\CB(\CH)$ and $\psi_0 = \omega_\eta$ with
$\eta = 
 \left( \begin{array}{c}
 0 \\
 1 \\
       \end{array} 
\right) \,
$ on $M_2$, which are obtained for $\lambda=0$.

The intution from physics behind our arguments is the following:
In the physics interpretation these states may be interpreted as ground states occupying the lowest energy levels of their systems and it is therefore a natural strategy to prepare $\varphi_0$
experimentally by repeated interactions with the state $\psi_0$ in the environment. This allows quanta to move out of the system into the environment but not in the other direction. We only have to verify this intuition mathematically to prove Theorem \ref{th:micromaser}.
\\{\phantom{x}} 
\\
Claim: $\varphi_0$ is universally $J$-preparable by the preparing sequence $(\theta_n)$ with
$\theta_n := \psi_n := \bigotimes^n_{j=1} \psi_0$.
First note that for all initial states $\sigma$ on $\CB(\CH)$ 
\[
(\sigma \otimes \theta_n) J_n(x)  = (\sigma \otimes \bigotimes^n_{j=1} \psi_0) J_n(x) = \sigma \circ T^n_{\psi_0}
\]
(cf. the introduction to Section \ref{section:stationary markov})
and we see that the claim is equivalent to the assertion that the vector state $\varphi_0$ is absorbing for  $T_{\psi_0}$. To prove this we can use Proposition \ref{prop:absorbing-criteria}(ii). In fact, note first that for $\lambda=0$ the classical transition matrix obtained by restriction to the diagonal algebra $\CD$ is given by
\[
T_{\CD,\psi_0} 
=
\left( \begin{array}{cccccc}
 |\alpha_0|^2=1 &       0           &          0        & 0                 & 0 &\ldots  \\
 |\beta_1|^2       & |\alpha_1|^2 &          0        & 0                 & 0 & \ldots  \\
0                      &  |\beta_2|^2  & |\alpha_2|^2 & 0                 & 0 &\ldots  \\
0                      &         0          & |\beta_3|^2  & |\alpha_3|^2 & 0 &\ldots  \\
 \vdots              &        \vdots   & \vdots         & \vdots         & \vdots &\ddots \\                                   
       \end{array} 
\right)\,.
\]

\noindent
This is a pure death process and so for all probability measures $\mu$ on $\Nset_0$
and any $N \in \Nset_0$ the sequence $\mu \circ T^n_{\CD,\psi_0}  \big(\{j \in \Nset_0 \colon 0 \le j \le N\} \big)$ is increasing with $n$. Clearly this implies that all probability measures 
on $\Nset_0$
are tight with respect to $T_{\CD,\psi_0}$. But then all normal states $\theta$ on $\CB(\CH)$ are tight for $T_{\psi_0}$: just use suitable finite dimensional projections $p \in \CD$ from subsets of the form $\{j \in \Nset_0 \colon 0 \le j \le N\}$ and the fact that for such projections
\[
\theta \circ T^n_{\psi_0}(p) = \theta|_\CD \circ T^n_{\CD,\psi_0}(p)\,.
\]
Further, for the support projection $p_{\delta_0}$ of the vector state $\varphi_0$ it is easy to check (for example by induction) from the explicit form of $T_{\CD,\psi_0}$ 
(cf. the analogous consideration in \cite{GKL06}, Section 6)
that the linear span of $T^n_{\psi_0}(p_{\delta_0}) = T^n_{\CD,\psi_0}(p_{\delta_0})$ over all $n \in \Nset_0$ is equal to the linear span of all the one-dimensional projections $p_{\delta_n}$ over all $n \in \Nset_0$, i.e., to a weak$^*$-dense subset of the diagonal algebra $\CD$ (here the absence of the trapping state condition, i.e., $\beta_n \not=0$ for all $n \ge 1$, is essential). 
So there does not exist a normal state $\theta$ on $\CB(\CH)$ which vanishes on $T^n_{\psi_0}(p_{\delta_0})$ for all $n$. 
Now it follows from Proposition \ref{prop:absorbing-criteria}(ii) that $\omega_{\delta_0}$ is absorbing for $T_{\psi_0}$ and hence universally $J$-preparable. This proves our claim.

To apply Theorem \ref{th:reverse} we also have to look at the 
reverse transition $J^r$. But an inspection of the definition of $J$ reveals that $J^r$ has the same form as $J$, only the phases in the sequences $(\alpha_n)$, $(\alpha^+_n)$, $(\beta_n)$, and $(\beta^+_n)$ have changed, hence the resulting classical transition matrices on the diagonal remain unaltered.
Therefore, all arguments used for $J$ also apply for $J^r$ and we conclude that $\omega_{\delta_0}$ is also universally $J^r$-preparable. This concludes the proof of the theorem.
\end{proof}

{\bf Acknowledgement:} 
Research was partially completed during a visit of B.\,K. in Aberystwyth in 2012 and while two of the authors (R.\,G. and B.\,K.) were visiting the Institute for Mathematical Sciences (IMS), National University of Singapore, for the program "Mathematical Horizons for Quantum Physics" in 2013. We gratefully acknowledge the corresponding funding by the EPSRC-Research Grant EP/G039275/1 and by the IMS. We thank A. G\"{a}rtner for suggesting a simplified proof for Proposition \ref{prop:orthogonal} based on Proposition \ref{prop:harmonic}. Finally we thank the referee for helpful remarks which led to some clarifications in our text, in particular to Remark \ref{rem:jaynes-cummings}.

\end{document}